\documentclass{SCIYA2017enOL}
\online
\begin{document}

\ensubject{fdsfd}

\ArticleType{ARTICLES}
\Year{2018}
\Month{January}%
\Vol{}
\No{1}
\BeginPage{1} %
\DOI{}
\ReceiveDate{January 11, 2018}
\AcceptDate{May 3, 2018}

\title[]{Numerical invariant tori of symplectic integrators for integrable Hamiltonian systems}
{Numerical invariant tori of symplectic integrators for integrable Hamiltonian systems}

\author[1,$\ast$]{Zhaodong Ding}{dingzhd@imu.edu.cn}
\author[2,3]{Zaijiu Shang}{zaijiu@amss.ac.cn}

\AuthorMark{Ding Z}


\address[1]{School of Mathematical Sciences, Inner Mongolia University, Hohhot {\rm010021}, China}
\address[2]{Institute of Mathematics, Academy of Mathematics and Systems Science,\\
Chinese Academy of Sciences, Beijing {\rm100190}, China}
\address[3]{Hua Loo-Keng Key Laboratory of Mathematics, Academy of Mathematics and Systems Science,\\
Chinese Academy of Sciences, Beijing {\rm100190}, China}

\abstract{In this paper, we study the persistence of invariant tori of integrable Hamiltonian systems satisfying R\"{u}ssmann's non-degeneracy condition when
symplectic integrators are applied to them. Meanwhile, we give an estimate of the measure of the set occupied by the invariant tori in the phase space. On an invariant torus, the one-step map of the scheme is conjugate to a one parameter family of linear rotations with a step size dependent frequency vector in terms of iteration. These results are a generalization of Shang's theorems (1999, 2000), where the non-degeneracy condition is assumed in the sense of Kolmogorov. In comparison, R\"{u}ssmann's condition is the weakest non-degeneracy condition for the persistence of invariant tori in Hamiltonian systems. These results provide new insight into the nonlinear stability of symplectic integrators.}



\maketitle

\section{Introduction}

An algorithm for numerically solving systems of ordinary differential equations is said
to be symplectic if its step-transition map is symplectic whenever the system is Hamiltonian. When applying a symplectic integrator to an integrable Hamiltonian system, the symplectic integrator can be written as a nearly integrable symplectic mapping with small twist where the time-step size is the perturbation parameter~\cite{Shang1999}. In other words, symplectic integrator may be characterized as a perturbation of the phase flow of the integrable system to which the integrator is applied. To some extent, the stability of symplectic integrators applied to integrable Hamiltonian systems may be related to the existence of invariant tori of nearly integrable symplectic mappings. The latter can be investigated in the setting of the well-known KAM (Kolmogorov-Arnold-Moser) theory. For instance, using the Moser's twist theorem, Sanz-Serna~\cite{Sanz1986} claimed the stability when the leapfrog scheme applied to pendulum dynamics with small enough step sizes. Combining the Moser's twist theorem and the theory of normal forms for Hamiltonian systems, Skeel and Srinivas~\cite{Skeel} gave a completely rigorous nonlinear stability analysis for area-preserving integrators with elliptic equilibria.

The classical KAM theorem states that when the frequency map
$\omega$ satisfies: (i) the non-degeneracy condition $\mbox{det}(\frac{\partial\omega(p)}{\partial p})\neq 0$ which means that the frequency map is a local diffeomorphism; (ii) the strong non-resonance conditions $|\left\langle k,\omega\right\rangle|\geq \frac{\gamma}{|k|^{\tau}}$, $\forall~k\in \mathbb{Z}^n\backslash
\{0\}$ with positive constants $\gamma$ and $\tau$ for a given frequency vector $\omega\in \mathbb{R}^n$(also known as Diophantine condition),
the corresponding torus persists with small deformation in the
perturbed integrable Hamiltonian systems if the size $\epsilon$ of
the perturbation is small enough. Later, R\"{u}ssmann announced~\cite{Ru1989} and proved~\cite{Ru2001} a generalized KAM theorem under a weaker non-degeneracy condition. This non-degeneracy condition was defined as follows. let $\mathcal{I}$ be an open and connected subset of $\mathbb{R}^n$ and $\omega:\mathcal{I}\rightarrow \mathbb{R}^n$ be a real analytic vector function, $\omega$ is called non-degenerate if the range $\omega(\mathcal{I})$ of $\omega$ does not lie in an $(n-1)$-dimensional linear subspace of $\mathbb{R}^n$, or equivalently, $\langle c,\omega\rangle\neq 0$ for any $c\in \mathbb{R}^{n}\backslash\{0\}$. Sevryuk~\cite{Sevryuk95} pointed out that R\"{u}ssmann's condition is not only sufficient, but also necessary for the existence of the perturbed tori in the analytic case. Further reviews and applications on the KAM theory can be referred to the monograph by Arnold et al.~\cite{Arnold2007} and the survey article by Sevryuk~\cite{Sevryuk03}.

For the discrete analogues to Hamiltonian systems (i.e., symplectic mappings), various KAM-type theorems have been established. Moser~\cite{Moser} first investigated the nearly integrable twist mapping on the annulus, and proved the existence of invariant curves by virtue of some intersection property and some non-degeneracy
condition. Using R\"{u}ssmann's non-degeneracy condition, Zhu et al.~\cite{Zhu} and Lu et al.~\cite{Lu2017} obtained the persistence of lower dimensional hyperbolic and elliptic invariant tori, respectively, for nearly integrable twist symplectic mappings. These results are similar to that for Hamiltonian systems~\cite{Graff,Po1989}, but the proofs are quite different because of respectively diverse structures~\cite{Lu2017}. In Shang's paper~\cite{Shang2000}, the existence of the highest dimensional tori was established for small twist symplectic mappings under the classical Kolmogorov non-degeneracy condition. Moreover, estimates for the bounds of the allowed perturbation and the relative measure of the complement of invariant tori in phase space are provided explicitly: the former is $O(\gamma^2\theta\Theta^{-2})$ and the latter $O(\gamma(\theta\Theta^{-1})^{-n})$, where $n$ is the degrees of freedom of the mappings; $\gamma$ is the Diophantine constant; $\theta$ and $\Theta$ are the nondegeneracy parameters of the frequency map and its inverse respectively (assumed that $\theta|p_1-p_2|\leq|\omega(p_1)-\omega(p_2)|\leq\Theta|p_1-p_2|$ which always holds locally).

By applying this kind of theorem, Shang~\cite{Shang1999} obtained a numerical version of KAM theroem for symplectic algorithms. More precisely, Let $\mathcal{X}$ be a completely integrable Hamiltonian system with $n$ degrees of freedom. Assume that $\mathcal{X}$ is analytic and nondegenerate in the sense of Kolmogorov. Then in the phase space of $\mathcal{X}$, there exists a Cantor family of $n$-tori such that the solutions of symplectic integrator applied to $\mathcal{X}$ are conjugate to a one parameter family of linear rotations on the preserved tori, provided a sufficiently small time-step $t$. These tori, which are called numerical invariant tori, possess all the standard properties in KAM theory: they are close to the unperturbed invariant $n$-tori of the system, they carry quasi-periodic motions and depend on the frequency vector in a Whitney-smooth way, the Lebesgue measure of the complement to their union tends to zero as $t\to 0$, etc. In addition, from Theorem 2 in~\cite{Shang1999}, it can be found that the preserved invariant tori have frequencies of the form $\omega_t=t\omega$ satisfying some Diophantine condition, where $t$ is the step size of the algorithm and $\omega$
belongs to the frequency domain of the system to which the algorithm is applied. Moreover, Shang~\cite{Shang00} showed that an invariant torus with any fixed Diophantine frequency can always be simulated very well by symplectic integrators for any step size in a Cantor set of positive Lebesgue measure near the origin for analytic non-degenerate integrable Hamiltonian systems.

Using the technique of including analytic symplectic maps in Hamiltonian flows~\cite{Benettin,Kuksin1994}, Moan~\cite{Moan} gave a proof on the existence of numerical invariant tori of symplectic algorithm when it applies to Hamiltonians with R\"{u}ssmann non-degeneracy. However, as pointed out by Sevryuk\footnote{MR2023432 (2004k: 37127) Moan P C. On the KAM and Nekhoroshev theorems for symplectic integrators and implications for error growth. Nonlinearity 17 (2004), 67--83. (Reviewer: Mikhail B. Sevryuk)}, R\"{u}ssmann non-degeneracy of a Hamiltonian with $n$ degrees of freedom does not imply the same non-degeneracy of a modified Hamiltonian with $n+1$ degrees of freedom (a detailed example was provided there), so the proof in~\cite{Moan} is not valid. In the present study, we give a direct proof on the existence of numerical invariant tori for Hamiltonian systems satisfying R\"{u}ssmann non-degeneracy condition by proposing a KAM-like theorem for small twist symplectic mappings. This generalizes Shang's results (1999, 2000) on the numerical KAM theorem of symplectic algorithms.

Unlike Kolmogorov's non-degeneracy assumed in~\cite{Shang1999}, where the invariant tori can be prescribed by frequency vectors, in R\"{u}ssmann's case the invariant tori may have drifted frequencies in the KAM iteration due to the weaker R\"{u}ssmann non-degeneracy. As a result, it may not be guaranteed that the set of frequencies of the preserved invariant tori completely corresponds to that of frequencies satisfying Diophantine condition of the original system due to frequency drift. That means not all the invariant torus with the fixed Diophantine frequency vector can be simulated well by the symplectic integrators for integrable Hamiltonian system with R\"{u}ssmann non-degeneracy. Fortunately, the set of frequencies of the preserved invariant tori tends to that of Diophantine frequencies of the original system as the step size $t\to 0$ (see \Cref{rm:3}). Based on this fact, it is reasonable to conjecture that "most" of the invariant tori of an analytic integrable Hamiltonian system with R\"{u}ssmann non-degeneracy can still be simulated well by symplectic integrators, provided an enough small and suitably chosen step size. Here there are some rather tricky and unsolved issues that need to be addressed. For example, how to select the "suitable" step sizes and what structure does the set of those admitted step sizes have? Is it possible to give an explicit measure estimate of the set of the invariant tori that can be simulated by symplectic integrators? These questions are interesting and important, and will be discussed later.

The outline of this paper is as follows. In \cref{sec:R condition}, we present the definition of R\"{u}ssmann non-degeneracy and its properties. Our main results are shown in \cref{sec:main}, but the proof will be postponed to \cref{sec:KAM-like}. \Cref{sec:experiments} is devoted to some numerical experiments, and the conclusions follow in \cref{sec:conclusions}.

\section{R\"{u}ssmann non-degeneracy}
\label{sec:R condition}

This kind of non-degeneracy condition is first proposed by R\"{u}ssmann in~\cite{Ru1989}. In order to adapt to the case of symplectic
mappings, the original definition requires some modification (see also
~\cite{Sevryuk95}).

\begin{definition}\label{def:R}
  A real function
$g=(g_1,\ldots,g_n):\mathcal {Y}\longmapsto \mathbb{R}^n$ defined in
a domain $\mathcal {Y}\subseteq\mathbb{R}^n$ is called
non-degenerate if $c_1g_1+\cdots c_ng_n\neq c_0$ for any vector of
constants $(c_0,c_1,\ldots,c_n)\in \mathbb{R}^{n+1}\backslash\{0\}$.
\end{definition}

From a geometric point of view, if a map $g$ satisfies the above definition, it means that the range $g(\mathcal {Y})$ of $g$ does not lie in any affine hyperplane in $\mathbb{R}^n$, while Kolmogorov's non-degeneracy implies this map is a local homeomorphism. For convenience, we also call the non-degeneracy, in the sense of R\"{u}ssmann, as the weak non-degeneracy.

\begin{remark}\label{rem:my-1}
Xu, You, and Qiu~\cite{Xu} proved that in the analytic case R\"{u}ssmann's non-degeneracy condition, i.e., $\langle c,g\rangle\neq 0$ for any $c\in \mathbb{R}^{n}\backslash\{0\}$, is equivalent to
\begin{equation}
\mbox{rank}\Big\{g,\frac{\partial^{\alpha}g}{\partial p^{\alpha}}~|~\forall~\alpha\in \mathbb{Z}^n,~|\alpha|\leq n-1\Big\}=n\,,
\end{equation}
where $|\alpha|=\alpha_1+\alpha_2+\cdots+\alpha_n$ and $\alpha_i\geq 0$. This can be easily extended to the situation in \cref{def:R}. Define $\tilde{g}=(g,1)\in \mathbb{R}^{n+1}$ analytic on $\mathcal {Y}$, then $\langle \tilde{c},\tilde{g}\rangle\neq 0$ for any $\tilde{c}=(c_0,c_1,\ldots,c_n)\in \mathbb{R}^{n+1}\backslash\{0\}$, is equivalent to
\begin{equation}
\mbox{rank}\Big\{\tilde{g},\frac{\partial^{\alpha}\tilde{g}}{\partial p^{\alpha}}~|~\forall~\alpha\in \mathbb{Z}^n,~|\alpha|\leq n-1\Big\}=n+1\,.
\end{equation}
\end{remark}

Now we introduce some notations used in this paper from~\cite{Ru2001}: Let $B\subseteq\mathbb{C}^n$ be open, and $g:B\rightarrow\mathbb{C}^m$ be a $\nu$-times continuously
differentiable function, denoted by $g\in C^{\nu}(B,\mathbb{C}^m)$.
As usual, the $\nu$-th derivative of $g$ in $x\in B$ is
denoted by $(a_1,\cdots,a_{\nu})\longmapsto$
$D^{\nu}g(x)(a_1,\cdots,a_{\nu}), a_j\in\mathbb{C}^n,
j=1,\ldots,\nu$. Moreover, we write
$D^{\nu}g(x)(a^{\nu}):=D^{\nu}g(x)(a,\cdots,a)$, $a~\in\mathbb{C}^n$,
\begin{equation}\nonumber
|D^{\nu}g(x)|_{2}:=\max_{a\in\mathbb{C}^n \atop
|a|_2=1}|D^{\nu}g(x)(a^{\nu})|_2, \quad \mbox{for all},~  x\in B.
\end{equation}
and
$
|D^{\nu}g|_{A} :=\sup\limits_{x\in A}|D^{\nu}g(x)|_2
$,
$|g|_{A}^{\nu}:=\max\limits_{0\leq\mu\leq\nu}|D^{\mu}g|_{A},\quad \mbox{for all}~ A\subseteq B$.

Assume $\omega:B\rightarrow \mathbb{R}^{n}$ be real analytic and weakly non-degenerate
function, $\chi(p)=(\omega(p),2\pi)$ and
$f(c,p)=\langle c,\chi(p)\rangle =\frac{\langle
k,\omega(p)\rangle+2\pi l}{|\widetilde{k}|}$, where
$c=\frac{\widetilde{k}}{|\widetilde{k}|}$, $\widetilde{k}=(k,l),
~\forall~ k\in \mathbb{Z}^n\backslash\{0\}~\mbox{and}~l\in
\mathbb{Z}$. Due to the non-degeneracy of $\omega$, $f(c,p)$ is not a constant in $B$. Moreover, we provide a slightly improved version of R\"{u}ssmann's Lemma 18.2 in~\cite{Ru2001}:

\begin{lemma}\label{lem:1}
  For any non-void compact set $\mathcal
{K}\subseteq B$ there are $\mu_0=\mu_0(\omega,\mathcal {K})\in \mathbb{Z}^{+}$ $(1\leq \mu_0\leq n-1)$ and $\beta=\beta(\omega,\mathcal {K})>0$ such that
\begin{equation}\label{amount}
\max_{0\leq\mu\leq\mu_0}|D^{\mu}f(c,p)|_2\geq\beta\,,~ \mbox{for
all}~c\in S,~p\in \mathcal {K},
\end{equation}
Here and in the sequel $D$ only refers to the variable $p$ whereas $c$ is considered a parameter, and $S=\{c\in \mathbb{R}^{n+1}~|~|c|=1\}$.
\end{lemma}

\begin{remark}\label{rem:my-2}
In the original lemma from R\"{u}ssmann, $\mu_0$ is only known as a positive integer. Here we can further determine its range $(1\leq \mu_0\leq n-1)$ by means of the equivalent formulation of the non-degeneracy mentioned in \Cref{rem:my-1}. In particular, if $\mu_0=1$, the condition \eqref{amount} is actually equivalent to the Kolmogorov non-degeneracy condition $\mbox{det}(\frac{\partial\omega}{\partial p})\neq 0$, where $\beta$ will be related to the value of this determinant.
\end{remark}

Due to the analytic nature of $\omega$, the numbers
\begin{equation*}
\beta(\omega,\mu,\mathcal{K}):=\min_{\substack{p\in \mathcal{K}\\ c\in S}}\max_{0\leq\nu\leq\mu}|D^{\nu}f(c,p)|_2
\end{equation*}
exist for any non-void compact set $\mathcal{K}$ and for all $\mu\in \mathbb{Z}^+$, and obviously,
\begin{equation*}
0\leq \beta(\omega,1,\mathcal{K})\leq \beta(\omega,2,\mathcal{K})\leq \ldots\,.
\end{equation*}
\Cref{lem:1} yields the existence of some $\mu_0\in \{1,2,\ldots,n-1\}$ such that $\beta(\omega,\mu_0,\mathcal{K})>0$. In other words, there is a
smallest positive integer with this property. As in~\cite{Ru2001}, $\mu_0=\mu_0(\omega,\mathcal {K})\in \mathbb{Z}^+$ is called the index of
non-degeneracy of $\omega$ with respect to $\mathcal {K}$ and
$\beta=\beta(\omega,\mathcal {K})>0$ the amount of non-degeneracy of
$\omega$ with respect to $\mathcal {K}$. The index and amount of
non-degeneracy will play an important role in the measure estimate of the numerical invariant tori.

\section{Main results}
\label{sec:main}

Consider an integrable analytic Hamiltonian system with $n$ degrees of freedom in canonical form
\begin{equation}\label{eq:basic}
\dot{x}=-\frac{\partial K}{\partial y}(x,y), \quad\dot
y=\frac{\partial K}{\partial x}(x,y), \qquad(x,y)\in D\,,
\end{equation}
where $D$ is a connected bounded open subset of $\mathbb{R}^{2n}$; a dot represents differentiation with respect to $t$ (time);
$K:D\rightarrow \mathbb{R}^1$ is the Hamiltonian. By Arnold-Liouville theorem~\cite{Arnold-2}, there exists a symplectic diffeomorphism $\Psi:B\times T^n\rightarrow
D$ such that under the action-angle coordinates $(p,q)$, the new Hamiltonian
$H(p)=K\circ\Psi(p,q)$, $(p,q)\in B\times T^n$,
only depends on $p$, and \eqref{eq:basic} takes the simple form
\begin{equation}\label{eq:basic-3}
\dot{p}=0, \quad \dot{q}=\omega(p)=\frac{\partial H}{\partial p}(p)\,.
\end{equation}
The phase flow of the system is just the one parameter group of
rotations $(p,q)\rightarrow (p,q + t\omega(p))$ which leaves every
torus $\{p\}\times T^n$ invariant. We assume that the frequency mapping $\omega$ is of weak non-degeneracy.

Define
\begin{equation}\label{k-defined}
\mathcal {K}:=\{p\in B~|~|p-\widetilde{p}|_2\geq 2\rho,~\mbox{for
all}~\widetilde{p}\in \partial B\},
\end{equation}
where $\rho$ $(0<\rho\leq 1)$ as a parameter, and $ \widetilde{D}:=\Psi(\mathcal {K}\times T^n)$. Obviously, $\mathcal
{K}$ is a compact subset of $B$. Thus, the index $\mu_0$ and amount $\beta$ of
non-degeneracy of $\omega$ with respect to $\mathcal {K}$ are well defined (see \cref{sec:R condition}). Here we state our main result:

\begin{theorem}\label{thm:bigthm}
Apply an analytic symplectic integrator to the system \eqref{eq:basic} where the frequency mapping satisfies the weak non-degeneracy. For any given real number $\kappa>1$ and sufficiently small
$\gamma >0$, if the time step $t$ of the symplectic integrator is sufficiently small, there exist Cantor subsets $\mathcal {K}_{\gamma,t}$ of $\mathcal {K}$ and $D_{\gamma,t}$ of $\widetilde{D}$, such that:

{\rm(i)} If the one-step map $G_K^t$ of the scheme is restricted to $D_{\gamma,t}$, then there exists a
$C^{\infty}$-symplectic conjugation $\Psi_t:\mathcal {K}_{\gamma,t} \times T^n \rightarrow
D_{\gamma,t}$, such that
\begin{equation*}
\Psi_t^{-1}\circ G_K^t\circ\Psi_t(p,~q)=(p,~q+t\omega_{t}(p))\,,
\end{equation*}
where $\omega_{t}$ is the frequency map defined on the Cantor set $\mathcal {K}_{\gamma,t}$.

{\rm(ii)} $m(D_{\gamma,t})\geq(1-c_1'\gamma^{\frac{1}{\mu_0}-\frac{1}{\kappa\mu_0}})m(\widetilde{D})$,
where $c_1'$ is a positive constant not depending on $\gamma$ and $t$; $\mu_0\in \mathbb{Z}^+$ is the index of
non-degeneracy of $\omega$ with respect to $\mathcal {K}$.
\end{theorem}

Note that the above $C^{\infty}$-mappings $\Psi_t$ and $\omega_t$ have to be understood in the sense of Whitney derivatives~\cite{Po}, because they are defined on Cantor-like sets. Here $\gamma$ is the Diophantine constant (see \eqref{condition-1}). In order to make the measure of the non-resonant set to be positive at each step of the KAM iteration process, we introduce the parameter $\kappa>1$ (see (\ref{myk})).

\begin{remark}\label{rm:2}
The conclusion (i) implies that the symplectic integrator $G_K^t$ has invariant $n$-tori forming a Cantor set $D_{\gamma,t}$ in phase space. Note that the relative measure of the complement of invariant tori (of the order $O(\gamma^{\frac{1}{\mu_0}-\frac{1}{\kappa\mu_0}})$) may be bigger than the one for Hamiltonians with Kolmogorov non-degeneracy (of the order
$O(\gamma)$~\cite{Shang1999}) due to $\mu_0>1$. In fact, one can choose sufficiently large $\kappa$ to make the measure be almost of order $O(\gamma^{\frac{1}{\mu_0}})$, which is consistent with the result from Xu, You, and Qiu~\cite{Xu} to some extent (see Remark 1.4 in~\cite{Xu}).
\end{remark}

\begin{remark}\label{rm:3}
Similar to Kolmogorov non-degeneracy case, the difference of the frequency of the numerical invariant solution and the exact one is of the accuracy $O(t^s)$ if the starting values of the numerical orbits and the exact ones are the same, i.e.,
\begin{equation*}
\parallel
\omega_{t}-\omega\parallel_{\alpha+1,\mathcal {K}_{\gamma,t}}\, \leq \,c_2'\gamma^{-\frac{2+\alpha}{\kappa\mu_0(\mu_0+1)}}t^s\,,
\end{equation*}
where $s$ is the order of the symplectic numerical scheme; $\alpha$ is a positive constant; $c_2'$ is a positive constant not depending on $\gamma$ and t; $\parallel\cdot\parallel$ refers to a norm in Whitney sense (see Shang~\cite{Shang1999}). This result can be derived by (\ref{ineq:omega-2}) in the proof of \Cref{thm:my-2}.
\end{remark}

In addition, as a corollary of \Cref{thm:bigthm}, we have the following results on the conservation of first integrals, and its proof can be obtained by similar techniques as in~\cite{Shang1999}.

\begin{corollary}\label{cor:1}
Under the assumptions of the above theorem, there exist $n$ functions $F_1^t,\cdots,F_n^t$ which are defined on
the Cantor set $D_{\gamma,t}$ and are of class $C^{\infty}$ in the sense of
whitney such that

{\rm (i)} $F_1^t,\cdots,F_n^t$ are functionally independent and in
involution;

{\rm (ii)} Every $F_j^t$, $j=1,\ldots,n$, is invariant under the difference
scheme and the invariant tori are just the intersection of the level
sets of these functions;

{\rm (iii)} $F_j^t,~j=1,\ldots,n$ approximate $n$ independent integrals $F_j,~j=1,\ldots,n$
of the original integral system, with the order of accuracy equal
to $t^s$ on $D_{\gamma,t}$, in the norm of the class $C^{\alpha}$
for any given $\alpha\geq 0$.
\end{corollary}

In order to prove \Cref{thm:bigthm}, we first propose a KAM-like theorem for small twist symplectic maps under the weak non-degeneracy condition in the next section.

\section{KAM theorem for small twist symplectic maps}
\label{sec:KAM-like}

Consider a one parameter family of analytic symplectic mapping
$S_t:(p,q)\rightarrow (\hat{p},\hat{q})$ with $S_0=\mbox{identity}$,
to be defined implicitly in phase space $B\times T^n$ by
\begin{equation}\label{eq}
\left\{ \begin{array}{lll}
     \hat{p}=p-t\frac{\partial H}{\partial q}(\hat{p},q)=p-t\frac{\partial h}{\partial
     q}(\hat{p},q),\\
     \hat{q}=q+t\frac{\partial H}{\partial
     q}(\hat{p},q)=q+t\omega(\hat{p})+t\frac{\partial
     h}{\partial\hat{p}}(\hat{p},q),
     \end{array}
     \right.
\end{equation}
with the analytic generating function $H(\hat{p},q)=H_0(\hat{p})+h(\hat{p},q)$, i.e., $H\in C^{\omega}(B\times T^n)$. Here
$\omega(\hat{p})=\frac{\partial H_0}{\partial \hat{p}}(\hat{p})$, $B$ is a bounded open set of $\mathbb{R}^n$ and
$T^n=\mathbb{R}^n/(2\pi \mathbb{Z}^n)$ is the usual torus.

Without loss of generality, assume the domain of $H$ can extend analytically to the complex domain:
\begin{equation*}
\Xi(r_0,\rho_0)=\Big\{(p,q)\in \mathbb{C}^{2n}~\lvert~ |p-p^*|_2<r_0, |\mbox{Im}~q|<\rho_0,~\mbox{with}~p^*\in B,~\mbox{Re}~q\in T^n\Big\}
\end{equation*}
for some $r_0>0$ and $\rho_0>0$.

Let $B+r_0:=\{p\in \mathbb{C}^n~|~|p-p^*|_2<r_0~\mbox{with some}~p^*\in
B\}$ be the complex extension of $B$. We assume the frequency mapping $\omega$ is weakly non-degenerate and satisfies:
\begin{equation}\label{condition-5}
|\omega(p_1)-\omega(p_2)|_2\leq \Theta |p_1-p_2|_2 \quad\mbox{for}~
p_1,p_2\in B+r_0,
\end{equation}
with $|p_1-p_2|_2\leq r_0$ and constant $\Theta>0$. Note $\omega$ may be irreversible because of the weak non-degeneracy. Define
\begin{equation}\label{condition-1}
\Omega_{\gamma,t}:=\Big\{\omega\in \mathbb{R}^n~|~|e^{i\langle
k,t\omega\rangle}-1|\geq
\frac{t\gamma}{|k|^{\tau}},~\forall~k\in\mathbb{Z}^n\backslash
\{0\}\Big\}\,,
\end{equation}
with positive parameters $\gamma$ and $\tau$, as the set of frequency vectors that satisfies the Diophantine
condition. The following theorem is a generalization of Shang's Theorem 2~\cite{Shang1999} under R\"{u}ssmann non-degeneracy condition.

\begin{theorem}\label{thm:my-2}
 Given a real number $\tau> (n+1)\mu_0$ and $\kappa>1$. For the $2n$ dimensional mapping $S_t$ defined above, there exists a constant $\delta_0>0$, depending only on $n,~\tau,~r_0$ and $\rho_0$, such that for any
$0<\gamma<\min(1,(\frac{1}{2}r_0\Theta)^{\kappa\mu_0(\mu_0+1)},\gamma_1)$,
where $\gamma_1$ is defined in \eqref{2}, if
\begin{equation}\label{es}
|h(p,q)|_{\Xi(r_0,\rho_0)}\leq \delta_0\,\widetilde{\gamma}^2\,\Theta^{-1},
\end{equation}
where $\widetilde{\gamma}:=\gamma^{\frac{1}{\kappa\mu_0(\mu_0+1)}}$, then there exist a Cantor set $\mathcal {K}_{\gamma,t}\subseteq\mathcal
{K}$, a mapping $\omega_{\gamma,t}:\mathcal
{K}_{\gamma,t}\rightarrow\Omega_{\gamma,t}$ of class $C^{\infty}$,
and a symplectic mapping $\Phi_t:\mathcal {K}_{\gamma,t}\times
T^n\rightarrow\mathbb{R}^n \times T^n$ of class
$C^{\infty,\omega}$, in the sense of Whitney, such that:

{\rm (i)} $\Phi_t$ is a conjugation between $S_t$ and $R_t$, i.e.,
$S_t\circ\Phi_t=\Phi_t\circ R_t$,
where $R_t$ is a rotation on $\mathcal
{K}_{\gamma,t}\times T^n$ with frequency mapping $t\omega_{\gamma,t}$,
{\rm i.e.}, $R_t(P,\,Q)=(P,\,Q+t\omega_{\gamma,t}(P))$.

{\rm (ii)} The measure of $\mathcal {K}_{\gamma,t}$ satisfies
\begin{equation}\label{m(k)-1}
m(\mathcal {K}_{\gamma,t})\geq
(1-c_1\gamma^{\frac{1}{\mu_0}-\frac{1}{\kappa\mu_0}}\beta^{-1-\frac{1}{\mu_0}})m(\mathcal
{K}),
\end{equation}
where $c_1$ is a positive
constant depending on $n,\,\tau,r_0,\rho_0$ and the domain
$\mathcal{K}$.
\end{theorem}

\begin{remark}\label{rm:4}  \Cref{thm:my-2} can be extended to the nonanalytic case (i.e., $h\in C^r(B\times T^n)$ with a sufficiently large positive integer $r$), which makes the conditions and the proof more complicated. Here we only present the analytic version for simplicity.
\end{remark}

\begin{remark}\label{rm:5}
Compared with the symplectic mapping satisfying Kolmogorov non-degeneracy where the bound of the relative measure of the resonant invariant tori in phase space is $O(\gamma(\theta\Theta^{-1})^{-n})$~\cite{Shang1999}, it is found that the dimension $n$ of the frequency map does not enter into the measure estimate directly. Instead, the index $\mu_0$ and the amount $\beta$ of non-degeneracy of $\omega$ are closely related to the measure estimate. For small $\mu_0$ as well as large $\beta$, it implies a large relative measure of invariant tori preserved by the symplectic mapping.
\end{remark}

\begin{remark}\label{rm:6}
Unlike Kolmogorov non-degenerate case, the weakly non-degenerate frequency $\omega$ is no longer a local homeomorphism, so we are not able to pick out the values of $\omega$ that we want in advance at each step of the KAM iteration process. That implies the conclusion (4) of Shang's Theorem 2 in~\cite{Shang1999} does not hold, and frequency drift happens in general.

In addition, it is worth noting that in the case of Kolmogorov non-degeneracy, an invariant torus with any fixed Diophantine frequency $\omega$ (i.e. $\omega\in \Omega_{\gamma,t}$) of an analytic non-degenerate integrable Hamiltonian system can always be simulated by symplectic algorithms for any step size $t$ in a Cantor set of positive Lebesgue measure near the origin (e.g. $t\in \mathcal{C}(\omega)\subset(0,t_0)$). While in the R\"{u}ssmann non-degeneracy case, the above fact will no longer be guaranteed due to the frequency drift. However, because of the property of the frequency approximation in \Cref{rm:3}, the authors conjecture that there exists a set $\mathcal{F}_{t_0}\subset (0,t_0)\times \Omega_{\gamma,t}$ of relatively large Lebesgue measure such that for every $(t^*,\omega_{t^*})\in \mathcal{F}_{t_0}$, the invariant torus of the system with the frequency $\omega_{t^*}$ can always be simulated by symplectic integrators with $t^*$ as the step size. Moreover, the projection of the set $\mathcal{F}_{t_0}$ to the step size direction has relative full measure at the origin.
\end{remark}

Our main result (\Cref{thm:bigthm}) can be considered as a corollary of \Cref{thm:my-2}. In fact, Shang~\cite{Shang1999} proved that by applying the analytic symplectic conjugation $\Psi$ associated with the Hamiltonian $K$, the one-step map $G_K^t$ of symplectic integrators can be expressed as a nearly integrable symplectic map $\widetilde{G}_K^t$ like \eqref{eq}, where the function $h$ is replaced by $t^sh^t$ (see Lemma 3.3 in~\cite{Shang1999}). Thus, if the step size $t$ is sufficiently small so that the condition \eqref{es} is satisfied, \Cref{thm:my-2} can be applied to $\widetilde{G}_K^t$, and we get the existence of $\mathcal {K}_{\gamma,t}$, $\omega_{\gamma,t}$ and $\Phi_t$. They satisfy: $(a)$ $\widetilde{G}_K^t\circ\Phi_t=\Phi_t\circ R_t$; $(b)$ the estimate $m(\mathcal {K}_{\gamma,t})\geq
(1-c_1\gamma^{\frac{1}{\mu_0}-\frac{1}{\kappa\mu_0}}\beta^{-1-\frac{1}{\mu_0}})m(\mathcal
{K})$.

Combining $\widetilde{G}_K^t=\Psi^{-1}\circ G_K^t\circ\Psi$ with $(a)$, we have $\Psi_t^{-1}\circ G_K^t\circ\Psi_t= R_t$ where $\Psi_t=\Psi\circ\Phi_t$. That proves the conclusion (i) of \Cref{thm:bigthm}. Notice that $m(\mathcal {K}_{\gamma,t}\times T^n)=m(\mathcal {K}_{\gamma,t})\cdot m(T^n)$, so one can obtain the measure estimate in conclusion (ii) of \Cref{thm:bigthm} with $c_1'=c_1\beta^{-1-\frac{1}{\mu_0}}$ by using the symplectic diffeomorphism characteristic of $\Psi$. Therefore, the proof of \Cref{thm:bigthm} is ultimately attributed to the one of \Cref{thm:my-2}. The proof of \Cref{thm:my-2} is summarized in the following outline.

\subsection{Outline of the proof of \Cref{thm:my-2}}
\label{subsec:proof}

In order to deal with the weak non-degeneracy, we essentially follow the same idea from R\"{u}ssmann~\cite{Ru2001}, i.e., separating the iteration process for the construction of invariant tori from the proof of the existence of enough non-resonant frequency vectors in KAM steps. However, because of the technical difference between Hamiltonian
system and symplectic mapping, appropriate changes are necessary. For example, for the construction of invariant tori, we employ Shang's technique (i.e., combining analytic function approximation with KAM iteration~\cite{Shang2000}), while for the existence of non-resonant frequency, we try to adapt R\"{u}ssmann's approach proposed for Hamiltonian systems to the case of symplectic mappings. To finish the proof of \Cref{thm:my-2}, it is important to combine these two aspects together and give suitable quantitative estimates.

As in~\cite{Shang2000}, we first transform the mapping $S_t$ by the
partial coordinates stretching $\sigma_{\rho}:(x,y)\rightarrow
(p,q)=(\rho x,y)$ and obtain a new one $T_t=\sigma_{\rho}^{-1}\circ
S_t\circ \sigma_{\rho}:(x,y)\rightarrow (\hat{x},\hat{y})$ to be
defined in the new phase space $B_{\rho}\times T^n$ by
\begin{equation}\label{eq:2}
\left\{ \begin{array}{ll}
     \hat{x}=x-t\frac{\partial F}{\partial y}(\hat{x},y)\,,\\
     \hat{y}=y+t\frac{\partial F}{\partial \hat{x}}(\hat{x},y)\,,
     \end{array}
     \right.
\end{equation}
where
$F(x,y)=F_0(x)+f(x,y):=\rho^{-1}H_0(\rho x)+\rho^{-1}h(\rho x,y)$,
and
$B_{\rho}=\rho^{-1}B=\{x\in \mathbb{R}^n~|~\rho x\in B\}$.

The frequency mapping of the integrable part associated
to the generating function $F$ turns into
$\widetilde{\omega}(x)=\partial F_0(x)$, $x\in B_{\rho}$.
Here $\widetilde{\omega}$ is also weakly non-degenerate and satisfies the condition
\begin{equation}\label{ineq:omega}
|\widetilde{\omega}(x_1)-\widetilde{\omega}(x_2)|_2\leq \rho
\,\Theta |x_1-x_2|_2\,,
\end{equation}
for $x_1,~x_2\in B_{\rho}+r_{\rho}$ with $|x_1-x_2|_2\leq r_{\rho}$ and $r_{\rho}=\rho^{-1}r$.

Note that the index of non-degeneracy of the frequency map does not change through this stretching transformation, i.e., $\mu_0(\tilde{\omega},\,\cdot)=\mu_0(\omega,\,\cdot)$, so we just write $\mu_0$ for simplicity. While, since $D^{\mu}\,\widetilde{\omega}(x)=\rho^{\mu}D^{\mu}\,\omega(p) $, the amount of non-degeneracy of the frequency map has the relation
\begin{equation}\label{beta}
\tilde{\beta}:=\beta(\widetilde{\omega},\,\cdot)\geq
\rho^{\mu_0}\beta(\omega,\,\cdot)=\rho^{\mu_0}\beta\,.
\end{equation}

As a result of the non-reversibility of the frequency mapping, we have to define a new set
$
\mathcal {K}_{\rho}:= \{x\in B_{\rho}^{*}~|~|x-\tilde{x}|_2\geq 1
~\mbox{for all}~ \tilde{x}\in \partial B_{\rho}^{*}\}
$
to replace the set $I_{\rho;\gamma}$ defined by (2.6) in~\cite{Shang2000}, where
\begin{equation}\label{B}
B_{\rho}^{*}:= \{x\in B_{\rho}~|~|x-\tilde{x}|_2\geq 1~\mbox{for
all}~ \tilde{x}\in \partial B_{\rho}\}\,,
\end{equation}
and $\partial B_{\rho}$ means the boundary of $B_{\rho}$.
Note that $\mathcal {K}_{\rho}=\rho^{-1}\mathcal {K}$ and we have
\begin{equation}\label{K}
(\mathcal {K}_{\rho}+1)\cap \mathbb{R}^n \subseteq
B_{\rho}^{*}\subseteq (B_{\rho}^{*}+1)\cap \mathbb{R}^n\subseteq
B_{\rho}.
\end{equation}

As in~\cite{Shang2000}, we approximate $f$ by a real
analytic functions series
$\{f_j\}_{j=0}^{\infty}$ defined on $\mathcal{U}_j$ with $f_0=0$,
i.e.,
$|f-f_j|_{B_{\rho}^{*}\times T^n}\rightarrow 0$ ($j\rightarrow\infty$), where
\begin{equation}
\mathcal{U}_j=B_{\rho}\times T^n+(4s_j,4s_j)
\end{equation}
is the complex extension of $B_{\rho}\times T^n$; $s_j=s_0\,4^{-j}$ for $j=0,1,2,\ldots$\,.

Associating with each $f_j$, we define a mapping $T_j:(x,y)\rightarrow (\hat{x},\hat{y})$ by
\begin{equation*}
\left\{ \begin{array}{ll}
     \hat{x}=x-t\frac{\partial F_j}{\partial y}(\hat{x},y)\,,\\
     \hat{y}=y+t\frac{\partial F_j}{\partial \hat{x}}(\hat{x},y)\,,
     \end{array}
     \right.
\end{equation*}
with $F_j(x,y)=F_0(x)+f_j(x,y)$, and $T_j$ converges to $T_t$ on $B_{\rho}\times T^n$. Using a similar KAM iteration process with Shang~\cite{Shang2000}, we can construct analytic symplectic transformations $\Phi_j$, analytic functions $F_0^{(j)}$ and integrable rotations $R_j:(x,y)\rightarrow (x,y+t\omega^{(j)}(x))$, which are defined on nested complex domains $\mathcal{V}_j=B_{\gamma,t}^{(j)}\times
T^n+(r_j,s_j)$ with $r_j=s_j^{\lambda}$ and the parameter $\lambda>\tau+1$, such that $\omega^{(0)}=\widetilde{\omega}$, $\omega^{(j)}=\partial F_0^{(j)}:B_{\gamma,t}^{(j)}\rightarrow \Omega_{\gamma,t}$, and for $\forall~ \alpha\geq 1$,
\begin{equation}\label{ineq:omega-2}
|\omega^{(j)}-\omega^{(j-1)}|_{\mathcal{V}_j} \leq
r_{j+1}^{\alpha}\cdot c_2|f|_{B_{\rho}\times T^n},\quad j=1,2,\ldots\,.
\end{equation}
Furthermore, as $j\rightarrow\infty$, the limits
\begin{equation*}
C_j=R_j^{-1}\circ\Phi_j^{-1}\circ T_j\circ\Phi_j \rightarrow \mbox{identity},\quad \Phi_j\rightarrow \widetilde{\Phi}_t, \quad R_j\rightarrow \widetilde{R}_t,
\end{equation*}
exist on $B_{\gamma,t}^{(\infty)}\times T^n:=\Big(\bigcap\limits_{j=0}^{\infty}B_{\gamma,t}^{(j)}\Big)\times T^n$, where $\widetilde{\Phi}_t$ and $\widetilde{R}_t$ are well defined on $B_{\gamma,t}^{(\infty)}\times T^n$.

Therefore in the limit we have $T_t\circ\widetilde{\Phi}_t=\widetilde{\Phi}_t\circ\widetilde{R}_t$ on $B_{\gamma,t}^{(\infty)}\times T^n$. Transforming the mapping $T_t$ back to $S_t$ by the stretching $\sigma_{\rho}$ and, meanwhile, transforming $\widetilde{\Phi}_t$ and $\widetilde{R}_t$ to, say, $\Phi_t$ and $R_t$, respectively, then we have
\begin{equation*}
S_t\circ\Phi_t=\Phi_t\circ R_t, \quad \mbox{on}\quad \mathcal {K}_{\gamma,t}\times T^n\,,
\end{equation*}
with
$
\mathcal {K}_{\gamma,t}=\rho B_{\gamma,t}^{(\infty)}
=\{x\in \mathbb{R}^n~|~\rho^{-1}x\in B_{\gamma,t}^{(\infty)}\}$. This is just the conclusion (i) of \Cref{thm:my-2}.

When the frequency map satisfies Kolmogorov non-degeneracy (i.e., it is a local homeomorphism), one may keep those non-resonant frequencies fixed at every step of the above approximation process, such that all sets $B_{\gamma,t}^{(j)}$ are non-empty, so is $B_{\gamma,t}^{(\infty)}$.
However, the non-emptiness of these sets will be non-trivial for the weak non-degeneracy frequency map since the frequency can't be fixed in advance (see also \Cref{rm:6}).

In order to make the above approximation process work, it is crucial to prove the set $B_{\gamma,t}^{(j)}$ is non-empty for each $j\in \mathbb{Z}^{+}$, where we define $B_{\gamma,t}^{(0)}=\mathcal {K}_{\rho}$\,,
\begin{equation}\label{myset:1}
B_{\gamma,t}^{(j+1)}=\{x\in B_{\gamma,t}^{(j)}~\big{|}~|e^{i\langle
k,t\omega^{(j)}(x) \rangle}-1|\geq \frac{t\gamma}{|k|^{\tau}},
~0<|k|\leq m_j\}\,,
\end{equation}
 for $j=0,1,\ldots$, where $t$ is a parameter and $m_j=(\frac{1}{r_j})^{\frac{1}{\tau+1}}$.

\begin{remark}\label{rm:7}
In the standard KAM iteration process, the frequencies not satisfying the Diophantine condition \eqref{condition-1} must be excluded at every step of the iteration process. However, here it is enough to exclude finite resonant frequencies rather than the whole. Thus each $B_{\gamma,t}^{(j)}$ is not a Cantor one, though $B_{\gamma,t}^{(\infty)}$ is.
\end{remark}

\begin{theorem}\label{clm:2}
For sufficiently small $\delta_0$ and $\gamma$, we have

{\rm (i)} $B_{\gamma,t}^{(0)}\supseteq B_{\gamma,t}^{(1)}\supseteq \cdots
\supseteq
B_{\gamma,t}^{(\infty)}:=\bigcap\limits_{j=0}^{\infty}B_{\gamma,t}^{(j)}\neq
\emptyset$;

{\rm (ii)} $
m(B_{\gamma,t}^{(\infty)})\geq m(\mathcal
{K}_{\rho})-\widetilde{M}\gamma^{\frac{1}{\mu_0}-\frac{1}{\kappa\mu_0}}\beta^{-1-\frac{1}{\mu_0}}$\,, with
\begin{displaymath}
\widetilde{M}=Ad^{n-1}\,(\frac{\pi}{2})
^{\frac{1}{\mu_0}}\frac{\mu_0!\,\Theta^{\mu_0+2}}{\pi}(n^{-\frac{1}{2}}+3d)(1+3\pi+|\widetilde{\omega}|_{B_{\rho}})
\sum\limits_{k\in
\mathbb{Z}^n\backslash \{0\}} |k|^{-\frac{\tau+1-\mu_0}{\mu_0}}\,,
\end{displaymath}
where $A:=3(2\pi e)^{\frac{n}{2}}(\mu_0+1)^{\mu_0+2}[(\mu_0+1)!]^{-1}$ and $d=\sup\limits_{x,y\in\mathcal {K}_{\rho}}|x-y|_2$.
\end{theorem}

The details of the proof for \Cref{clm:2} will be given in the subsequent section. Now we derive the conclusion (ii) of \Cref{thm:my-2} taking this theorem for granted. Put
\begin{equation*}
{K}_{\gamma,t}=\rho B_{\gamma,t}^{(\infty)}
=\{x\in \mathbb{R}^n~|~\rho^{-1}x\in B_{\gamma,t}^{(\infty)}\}\,,
\end{equation*}
so $\mathcal
{K}_{\gamma,t}$ a non-empty Cantor subset of $\mathcal {K}$. Note $m(\mathcal
{K}_{\rho})=\rho^{-1}m(\mathcal {K})$ and $m(B_{\gamma,t}^{(\infty)})=\rho^{-1}m(\mathcal {K}_{\gamma,t})$, after inserting them into the conclusion (ii) of \Cref{clm:2}, we have
\begin{equation*}
\begin{split}
m(\mathcal {K}_{\gamma,t})&\geq m(\mathcal
{K})-\rho\widetilde{M}\gamma^{\frac{1}{\mu_0}-\frac{1}{\kappa\mu_0}}\beta^{-1-\frac{1}{\mu_0}} \\
 &= \Big(1-c_1
\gamma^{\frac{1}{\mu_0}-\frac{1}{\kappa\mu_0}}\beta^{-1-\frac{1}{\mu_0}}\Big)m(\mathcal
{K})\,,
 \end{split}
 \end{equation*}
where $c_1=\frac{\rho \widetilde{M}}{m(\mathcal {K})}$. This is the conclusion (ii) of \Cref{thm:my-2}.

\subsection{Proof of \Cref{clm:2}}
\label{subsec:proof-2}

First, we choose
$l=l(t\widetilde{\omega},k)\in \mathbb{Z}$ such that
\begin{equation*}
\Big|\frac{\langle k,t\widetilde{\omega}(x)\rangle+2\pi l
}{2}\Big|\in [0,\,\frac{\pi}{2}]\quad \mbox{for}~\forall~k\in \mathbb{Z}^n \backslash
\{0\},~\mbox{and}~ x\in B_{\rho}.
\end{equation*}
Thus,
\begin{displaymath}
|e^{i\langle
k,t\widetilde{\omega}(x) \rangle}-1|=2\sin{\Big|\frac{\langle
k,t\widetilde{\omega}(x)\rangle+2\pi l }{2}\Big|}\geq 2\cdot
\frac{2}{\pi}\cdot \Big|\frac{\langle
k,t\widetilde{\omega}(x)\rangle+2\pi l }{2}\Big|,
\end{displaymath}
and so
\begin{displaymath}
\Big\{x\in \mathcal {K}~\big{|}~|e^{i\langle
k,t\widetilde{\omega}(x) \rangle}-1|\geq
\frac{t\gamma}{|k|^{\tau}},~k\neq 0\Big\} \supseteq \Big\{x\in
\mathcal {K}~\big{|}~|\langle k,t\widetilde{\omega}(x) \rangle+2\pi
l|\geq \frac{\pi}{2}\,\frac{t\gamma}{|k|^{\tau}},~k\neq 0\Big\}
\end{displaymath}
for any subset $\mathcal {K}\subseteq B_{\rho}$. Denote $\mathcal
{K}_{\gamma,t}^{(0)}=\mathcal {K}_{\rho}=B_{\gamma,t}^{(0)}$,
\begin{equation}\label{tilde-k}
\mathcal {K}_{\gamma,t}^{(j+1)}=\{x\in \mathcal
{K}_{\gamma,t}^{(j)}~\big{|}~ |\langle k,t\omega^{(j)}(x)
\rangle+2\pi l(t\omega^{(j)},k)|\geq
\frac{\pi}{2}\,\frac{t\gamma}{|k|^{\tau}},~0<|k|\leq m_j\}.
\end{equation}
Thus $B_{\gamma,t}^{(j)}\supseteq \mathcal {K}_{\gamma,t}^{(j)}$ ($j=0,1,\ldots$),
and we only need to prove $\mathcal
{K}_{\gamma,t}^{(\infty)}=\bigcap\limits_{j=0}^{\infty}\mathcal
{K}_{\gamma,t}^{(j)}\neq\emptyset$ and the corresponding measure
estimate.

Now we introduce some notations used in this paper from~\cite{Ru2001}. Denote $\chi_j=(t\omega^{(j)},\,2\pi)\in
\mathbb{R}^{n+1}$. A pair $\mathcal
{L}_j:=(\mathcal {K}_{\gamma,t}^{(j)},\chi_j)$ is a link if $\mathcal
{K}_{\gamma,t}^{(j)}\neq\emptyset$ and
$\chi_j=(t\omega^{(j)},\,2\pi)\in C^{\omega}(P_j,\,\mathbb{C}^n)$
with $P_j=\mathcal {K}_{\gamma,t}^{(j)}+r_{j+1}\subseteq
\mathbb{C}^n$. A link $\mathcal {L}_j$ is open if $\mathcal
{K}_{\gamma,t}^{(j+1)}\neq \emptyset$. The initial link $\mathcal {L}_0=(\mathcal
{K}_{\gamma,t}^{(0)},\chi_0)=(\mathcal
{K}_{\rho},(t\omega^{(0)},2\pi))$. If $\mathcal {L}_j$ is well
defined and $\mathcal {K}_{\gamma,t}^{(j+1)}\neq \emptyset$, $\mathcal {L}_{j+1}$ can be defined recursively as follows:
\begin{equation}\label{chain}
\chi_{j+1}=(\chi_j+\triangle\chi_j)|_{P_{j+1}}, \quad
P_{j+1}=\mathcal {K}_{\gamma,t}^{(j+1)}+r_{j+2}.
\end{equation}
A collection of links can be called a chain and
write
\[
\langle \mathcal {L}_j\rangle_{0\leq j\preceq \nu} =\left\{
\begin{array}{ll}
\langle \mathcal {L}_0,\mathcal {L}_1,\cdots,\mathcal {L}_{\nu}\rangle &\mbox{for}~\nu<\infty,\\
\langle \mathcal {L}_0,\mathcal {L}_1,\cdots\rangle
&\mbox{for}~\nu=\infty,
\end{array}
\right.
\]
where the symbol $j\preceq\nu~\Longleftrightarrow~(j\leq\nu<\infty
~\mbox{or}~j<\nu=\infty)$. A chain $\langle \mathcal {L}_j\rangle_{0\leq j\preceq \nu}$ is called maximal if either $\nu<\infty$ and $\mathcal {L}_{\nu}$ is not open
or $\nu=\infty$.

In our situation, we define
\begin{equation}\label{norm}
|\chi_j|=|t\omega^{(j)}|_2+2\pi, \quad
|\triangle\chi_j|=|t\omega^{(j)}|_2\,,
\end{equation}
with
$\triangle\chi_j=(t\triangle\omega^{(j)},0)=(t\omega^{(j+1)}-t\omega^{(j)},0)\in
C^{\omega}(P_j,\mathbb{C}^{n+1})$. According to
(\ref{ineq:omega-2}), $\triangle\chi_j$ satisfies the estimate
\begin{equation} \label{eq:11}
|\triangle\chi_j|_{P_j}=|t\triangle\omega^{(j)}|_{P_j}\leq r_{j+1}^{\alpha+1}\cdot \,c_5\,t\widetilde{\gamma}\,\delta_a=r_{j+1}^{\alpha+1}\cdot t\,L_0\,,
 \end{equation}
where $t\in [0,1]$ is a parameter and
$L_0=\,c_5\,\widetilde{\gamma}\delta_a$. Denote $L_j=L_0\cdot
r_{j+1}^{\alpha+1}$, thus $|\triangle\chi_j|_{P_j}\leq t \,L_j$.

Denote $\widetilde{P}=(\mathcal {K}_{\rho}+1)\cap \mathbb{R}^n$. We
define the extension $\widetilde{\mathcal {L}}_0$ of the initial
link $\mathcal {L}_0=(\mathcal {K}_{\gamma,t}^{(0)},\chi_0)$ by
$\widetilde{\mathcal {L}}_0=(\mathcal {K}_{\gamma,t}^{(0)},\chi_0,\tilde{\chi}_0)$\,,
where $\tilde{\chi}_0=\chi|_{\tilde{P}}$~. In particular, we have
$\tilde{\chi}_0|_{\mathcal {K}_{\gamma,t}^{(0)}}=\chi_0|_{\mathcal
{K}_{\gamma,t}^{(0)}}$. For $\nu >0$ we consider a chain $\langle
\mathcal {L}_j\rangle_{0\leq j\preceq \nu}$ and functions
$\triangle\chi_j=(t\triangle\omega^{(j)},0)\in C^{\omega}(P_j,\mathbb{C}^{n+1})$\,, for $0\leq j< \nu$.
According to Theorem 19.7 in~\cite{Ru2001}, these functions $\triangle\chi_j$ can be attached to the $C^{\infty}$--functions
$\widetilde{\triangle\chi}_j=(t\widetilde{\triangle\omega}^{(j)},0)\in C^{\infty}(\widetilde{P},\mathbb{R}^{n+1})$ for $0\leq j<\nu$,
with the estimates
\begin{equation}\label{2.2.30}
|D^{\mu}\widetilde{\triangle\chi}_j|_{\widetilde{P}}\leq
c(n,\mu)\,r_{j+1}^{-\mu}t\,L_j, \quad
\mu=0,1,\ldots
\end{equation}
such that
$\widetilde{\triangle\chi}_j|_{\mathcal
{K}_{\gamma,t}^{(j)}}=\triangle\chi_j|_{\mathcal
{K}_{\gamma,t}^{(j)}}$\,, where $c(n,\mu)$ is defined by (19.10) in~\cite{Ru2001}.

Now we define $C^{\infty}$--functions on $\widetilde{P}$ recursively by
\begin{equation}\label{recursively}
\widetilde{\chi}_0=\chi|_{\widetilde{P}},\quad
\widetilde{\chi}_{j+1}=\widetilde{\chi}_j+\widetilde{\triangle\chi}_j
,\quad 0\leq j< \nu\,,
\end{equation}
such that we obtain $\widetilde{\chi}_{j}|_{\mathcal
{K}_{\gamma,t}^{(j)}}=\chi_{j}|_{\mathcal {K}_{\gamma,t}^{(j)}},
~0\leq j\preceq \nu$.

By this way, each link $\mathcal {L}_j$
of the chain $\langle \mathcal {L}_j\rangle $ has a recursively well
defined extension $\widetilde{\mathcal {L}}_j=\langle \mathcal
{K}_{\gamma,t}^{(j)}\,,\chi_j\,,\widetilde{\chi}_j\rangle$. It makes
sense to call
\[
\langle \widetilde{\mathcal {L}}_j\rangle_{0\leq j\preceq \nu}
=\left\{
\begin{array}{ll}
\langle \widetilde{\mathcal {L}}_0,\cdots,\widetilde{\mathcal {L}}_{\nu}\rangle &\mbox{for}~\nu<\infty,\\
\langle \widetilde{\mathcal {L}}_0,\widetilde{\mathcal
{L}}_1,\cdots\rangle &\mbox{for}~\nu=\infty.
\end{array}
\right.
\]
the extension of the chain $\langle \mathcal {L}_j\rangle_{0\leq
j\preceq \nu}$.

The following \Cref{Rlem:1,Rlem:2,Rlem:3,Rlem:4} are the analogues of Lemma 20.3, Lemma 20.4, Theorem 17.1 and Theorem 18.5 in~\cite{Ru2001}, respectively, though some modifications and simplifications have been made for our situations. In particular, one must pay attention to the presence of step size $t$.

\begin{lemma}\label{Rlem:1}
Let $0\leq \mu<\alpha+1$ and $\langle \mathcal
{L}_j\rangle_{0\leq j\preceq \nu}$ be a chain with its extension
$\langle \widetilde{\mathcal {L}}_j\rangle_{0\leq j\preceq \nu}$.
Then the estimates
\begin{equation}\label{Cauchy}
|D^{\mu}(\widetilde{\chi}_j-\widetilde{\chi}_\sigma)|_{\widetilde{P}}\leq
c(n,\mu)\cdot
tL_0\,r_0^{\alpha-\mu+1}\frac{4^{-\lambda(\sigma+1)(\alpha-\mu+1)}-
4^{-\lambda(j+1)(\alpha-\mu+1)}}{1-4^{-\lambda(\alpha-\mu+1)}}\,,
\end{equation}
\begin{equation}\label{chi_j}
|\widetilde{\chi}_j-\widetilde{\chi}_0|_{\widetilde{P}}^{\mu}\leq
c(n,\mu)\cdot
tL_0\,r_0^{\alpha-\mu+1}\frac{4^{-\lambda(\alpha-\mu+1)}}{1-4^{-\lambda(\alpha-\mu+1)}}\,,
\end{equation}
and
\begin{equation}\label{chi}
|\widetilde{\chi}_j|_{\widetilde{P}}^{\mu}\leq
\mu!C_{\cdot}+c(n,\mu)\cdot
tL_0\,r_0^{\alpha-\mu+1}\frac{4^{-\lambda(\alpha-\mu+1)}}{1-4^{-\lambda(\alpha-\mu+1)}}
\end{equation}
hold for $0\leq\sigma\leq j\preceq\nu$. Here
$C_{\cdot}=2\pi+|\widetilde{\omega}|_{B_{\rho}}$\,, $c(n,\mu)$ is a constant only dependent on $n$ and $\mu$. In particular, $c(n,0)=1$.
\end{lemma}

\begin{proof}
The proof is similar to that of Lemma 20.3 in~\cite{Ru2001}. So we omit it.
\end{proof}

\begin{lemma}\label{Rlem:2}
Let
\begin{equation}\label{L_0}
L_0\leq \frac{1-4^{-\lambda(\alpha+1)}}{(r_0\cdot
4^{-\lambda})^{\alpha+1}}
\end{equation}
and an extended chain $\langle \widetilde{\mathcal
{L}}_j\rangle_{0\leq j\leq\nu<\infty}$ be given. Then the estimates
$|\widetilde{\chi}_j|_{\widetilde{P}}\leq C_{\cdot}+1$ $(0\leq j\leq \nu)$
are valid. Furthermore, for any function $\hat{\chi}=(t\hat{\omega},2\pi)\in
C^{\mu_0}(\widetilde{P},\mathbb{R}^{n+1})$ with the estimates
$|\widehat{\chi}|_{\widetilde{P}}\leq C_{\cdot}+1$ and $|\widehat{\chi}-\widetilde{\chi}_j|_{\widetilde{P}}\leq \frac{\pi}{2}t\gamma\, r_j$,
we have
\begin{equation}
\mathcal {K}_{\rho}=\mathcal {K}_{\gamma,t}^{(0)}\supseteq\mathcal {K}_{\gamma,t}^{(1)}\supseteq\cdots
\mathcal {K}_{\gamma,t}^{(\nu)}\supseteq \mathcal
{K}_{\gamma,t}^{(\nu+1)}\supseteq \bigcap_{|k|\leq m_{\nu}}\mathcal
{H}_{k}(\widehat{\chi})\,,
\end{equation}
where $\mathcal
{H}_{k}(\widehat{\chi})=\{x\in \mathcal {K}_{\rho}~\big{|}~|\langle
k,t\widehat{\omega} \rangle+2\pi l(t\widehat{\omega},k)|\geq
\pi\,\frac{t\gamma}{|k|^{\tau}},~k\neq 0\}$ and
$m_{\nu}=r_{\nu}^{-\frac{1}{\tau+1}}$.
\end{lemma}

\begin{proof}
The proof is deferred to \ref{sec:proof:1}.
\end{proof}

\begin{lemma}\label{Rlem:3}
Let $\mathcal {K}\subseteq\mathbb{R}^n$ be a
compact set with diameter $d=d(\mathcal
{K}):=\sup\limits_{x,y\in\mathcal {K}}|x-y|_2
> 0$. Define $B=(\mathcal {K}+\theta)\cap\mathbb{R}^n$ for some $\theta>0$,
and $g\in C^{\mu_0+1}(B,\mathbb{R})$ be a function with
\begin{equation}
\min_{y\in \mathcal {K}}\max_{0\leq\mu\leq\mu_0}|D^{\mu}g(y)|_2\geq
\beta\,,
\end{equation}
for some $\mu_0\in \mathbb{Z}^{+}$ and $\beta>0.$ Then for any
function $\tilde{g}\in C^{\mu_0}(B,\mathbb{R})$ satisfying
$|\tilde{g}-g|_{B}^{\mu_0}\leq \beta/2$, we have the
estimate
\begin{equation}
m(\{y\in \mathcal {K}\,\big{|}\,|\tilde{g}(y)|\leq \varepsilon\})\leq
A\,d^{n-1}\big(n^{-\frac{1}{2}}+2d+\theta^{-1}d\big)\big(\frac{\varepsilon}{\beta}\big)^{\frac{1}{\mu_0}}\,
\frac{1}{\beta}\,\max_{0<\mu\leq \mu_0+1}|D^{\mu}g|_{B},
\end{equation}
whenever $0<\varepsilon\leq \frac{\beta}{2\mu_0+2}$~, where
$A=3(2\pi e)^{\frac{n}{2}}(\mu_0+1)^{\mu_0+2}[(\mu_0+1)!]^{-1}.$
\end{lemma}

\begin{proof}
The proof is similar to that of Theorem 17.1 in~\cite{Ru2001}. So we omit it.
\end{proof}

\begin{lemma}\label{Rlem:4}
Suppose: {\rm (i)}  $\mathcal {K}\subseteq B_{\rho}^{*} ~(see ~\eqref{B})$ is a compact set with $m(\mathcal {K})>0$ such that $d=\sup\limits_{x,y\in
\mathcal {K}}|x-y|_2>0$;

{\rm (ii)} $\theta$ with $0<\theta\leq \min(1,\,dist(\mathcal
{K},\mathbb{R}^n\backslash B_{\rho}^{*}))$ is given and
$\widetilde{B}=(\mathcal {K}+\theta)\cap\mathbb{R}^n$;

{\rm (iii)} a real analytic function $\chi=(t\omega,2\pi)\in
C^{\omega}(B_{\rho}^{*},\mathbb{R}^{n+1})$ and a function
$\widetilde{\chi}=(t\widetilde{\omega},2\pi)\in
C^{\mu_0}(\widetilde{B},\mathbb{R}^{n+1})$ satisfying the estimates
\begin{equation}\label{three-ineq}
|\widetilde{\chi}|_{\widetilde{B}}\leq M_0, \quad
\max_{0<\mu\leq\mu_0+1}|D^{\mu}\chi|_{\widetilde{B}}\leq tM_1, \quad
|\chi|_{\widetilde{B}}-\widetilde{\chi}|_{\widetilde{B}}^{\mu_0}\leq
\frac{t\beta}{2}\,,
\end{equation}
where the norm is defined in \eqref{norm} and
$\mu_0$ and $\beta$ are the index and amount of $\omega$ with respect to $\mathcal
{K}$, respectively.

Then the measure of the set
\begin{equation*}
\mathcal {H}(\widetilde{\chi})=\Big\{x\in \mathcal {K}~\big{|}~|\langle k,t\widetilde{\omega}
\rangle+2\pi l(t\widetilde{\omega},k)|\geq \frac{\pi}{2}\,
\frac{t\gamma}{|k|^{\tau}},\quad \forall~k\in\mathbb{Z}^n\backslash
\{0\}\Big\}
\end{equation*}
can be estimated as
$
m(\mathcal {H}(\widetilde{\chi}))\geq m(\mathcal
{K})-\widetilde{M}\gamma^{\frac{1}{\mu_0}}/\beta^{1+\frac{1}{\mu_0}},
$
whenever
\begin{equation}\label{r}
0<\gamma\leq\Big(\frac{m(\mathcal
{K})\beta^{1+\frac{1}{\mu_0}}}{\widetilde{M}}\Big)^{\mu_0},
\end{equation}
where
$\widetilde{M}=A\,d^{n-1}(n^{-\frac{1}{2}}+2d+\theta^{-1}d)\,\Big(\frac{\pi}{2}\Big)
^{\frac{1}{\mu_0}}\,\Big(\frac{M_0}{\pi}+1\Big)M_1 \sum\limits_{k\in
\mathbb{Z}^n\backslash \{0\}} |k|^{-\frac{\tau+1-\mu_0}{\mu_0}}~.
$
\end{lemma}

\begin{proof}
The proof is deferred to \ref{sec:proof:2}.
\end{proof}

Using the above lemmas, we can get the properties of the chain $\langle \mathcal {L}_j\rangle_{0\leq j\preceq \nu}$ defined by \eqref{chain}, and give the corresponding measure estimate.
\begin{theorem}\label{thm:my-3}
Suppose
\begin{equation}\label{1}
\alpha>\mu_0-1, \qquad L_0\leq\min\Big(1,~
\frac{\tilde{\beta}}{2c(n,\mu_0)}\Big)\,,
\end{equation}
where $\tilde{\beta}=\beta(\widetilde{\omega},\mathcal {K}_{\rho})$,
and
\begin{equation}\label{2}
0<\gamma<\gamma_1:=\Big(\frac{m(\mathcal
{K}_{\rho})\beta^{1+\frac{1}{\mu_0}}}{\widetilde{M}\Theta^{\mu_0+1}}\Big)^{\mu_0}.
\end{equation}
Then the maximal chain defined by \eqref{chain} is infinite. Furthermore, for the infinite chain $\langle \mathcal {L}_j\rangle_{0\leq
j<\infty}$ with $\mathcal {L}_j=(\mathcal {K}_{\gamma,t}^{(j)},\chi_j)$,
we have
$\mathcal {K}_{\gamma,t}^{(\infty)}=\bigcap\limits_{j=0}^{\infty}\mathcal
{K}_{\gamma,t}^{(j)}\neq \emptyset$ and $m(\mathcal
{K}_{\gamma,t}^{(\infty)})\geq m(\mathcal
{K}_{\rho})-\widetilde{M}\gamma^{\frac{1}{\mu_0}-\frac{1}{\kappa\mu_0}}\beta^{-1-\frac{1}{\mu_0}}$.
\end{theorem}

\begin{proof}
We first prove $\mathcal {K}_{\gamma,t}^{(\infty)}\neq \emptyset$. If it is not so, there exists a positive integer $\nu$ such that $\langle \mathcal {L}_j\rangle_{0\leq j\leq
\nu}$ is a maximal. Assume $\langle \widetilde{\mathcal
{L}}_j\rangle_{0\leq j\leq \nu}$ is its extension.
For applying
\Cref{Rlem:4} with $\chi=(t\widetilde{\omega},2\pi)$,
$\widetilde{\chi}=\widetilde{\chi}_{\nu}$\,, $
\widetilde{B}=\widetilde{P},~\theta=1,~\mathcal {K}=\mathcal
{K}_{\rho}$, we must check the condition (\ref{three-ineq}).
Observe that $|\chi|_{\widetilde{P}}\leq C_{\cdot}$~. Using
Cauchy's estimate and (\ref{ineq:omega}) we have
\begin{equation*}
\max_{0<\mu\leq\mu_0+1}|D^{\mu}\chi|_{\widetilde{P}}\leq
\mu_0!\,t\Theta.
\end{equation*}
The condition (\ref{1}) permits the application of \Cref{Rlem:1} with
$j=\nu$ and $\mu=\mu_0$, so that we obtain
\begin{equation}\label{es:1}
|\widetilde{\chi}_{\nu}-\widetilde{\chi}_0|_{\widetilde{P}}^{\mu_0}\leq
c(n,\mu_0)\cdot
tL_0\,r_0^{\alpha-\mu_0+1}\frac{(\frac{1}{4})^{\lambda(\alpha-\mu_0+1)}}{1-(\frac{1}{4})^{\lambda(\alpha-\mu_0+1)}}~.
\end{equation}
Furthermore, by using $0<r_0<1$, $1-(\frac{1}{4})^{\lambda
(\alpha-\mu_0+1)}\geq (\frac{1}{4})^{\lambda (\alpha-\mu_0+1)}$ and the condition (\ref{1}), the bound in \eqref{es:1} becomes
\begin{equation*}
|\widetilde{\chi}_{\nu}-\widetilde{\chi}_0|_{\widetilde{P}}^{\mu_0}\leq t\tilde{\beta}/2.
\end{equation*}
Finally, by means of $L_0\leq 1$ and $1-(\frac{1}{4})^{\lambda
(\alpha+1)}\geq (\frac{1}{4})^{\lambda (\alpha+1)}$, we have
\begin{equation*}
L_0\leq \frac{1-4^{-\lambda (\alpha+1)}}{(r_0\cdot4^{-\lambda})^{\alpha+1}}~.
\end{equation*}

Then we can apply \Cref{Rlem:2} to get
\begin{equation}\label{es:2}
|\widetilde{\chi}_{\nu}|_{\widetilde{P}}\leq C_{\cdot}+1\,.
\end{equation}
Therefore, we get (\ref{three-ineq}) with $M_0=C_{\cdot}+1$ and
$M_1=\mu_0!\Theta$. According to \Cref{Rlem:4}, we have
\begin{eqnarray}
m(\mathcal {H}(\widetilde{\chi}_{\nu}))&\geq&m(\mathcal
{K}_{\rho})-\widetilde{M}\cdot\gamma^{\frac{1}{\mu_0}}\cdot
\tilde{\beta}^{-1-\frac{1}{\mu_0}}\nonumber\\
&\geq &m(\mathcal
{K}_{\rho})-\widetilde{M}\cdot\Theta^{\mu_0+1}\cdot\gamma^{\frac{1}{\mu_0}-\frac{1}{\kappa\mu_0}}\cdot
\beta^{-1-\frac{1}{\mu_0}}\,,\label{myk}
\end{eqnarray}
where we have used (\ref{beta}) and
$\rho=\gamma^{\frac{1}{\kappa\mu_0(\mu_0+1)}}\Theta^{-1}$. When $\gamma<\gamma_1$,
$m(\mathcal {H}(\widetilde{\chi}_{\nu}))>0$ holds with
\begin{equation*}
\mathcal {H}(\widetilde{\chi}_{\nu})=\Big\{x\in \mathcal
{K}_{\rho}~\big{|}~|\langle k,t\widetilde{\omega}^{(\nu)}
\rangle+2\pi l(t\widetilde{\omega}^{(\nu)},k)|\geq \frac{\pi}{2}\,
\frac{t\gamma}{|k|^{\tau}},\quad \forall~k\in\mathbb{Z}^n\backslash
\{0\}\Big\}~.
\end{equation*}

From (\ref{chi_j}) with $\mu=0,~j=\nu,~\sigma=j$, we have
\begin{equation}
|\widetilde{\chi}_{\nu}-\widetilde{\chi}_j|_{\widetilde{P}}\leq tL_0\cdot r_0^{\alpha+1}
\frac{4^{-\lambda(j+1)(\alpha+1)}-4^{-\lambda(\nu+1)(\alpha+1)}}{1-4^{-\lambda(\alpha+1)}}\leq
tL_0r_0^{\alpha+1}4^{-\lambda j(\alpha+1)}\,.
\end{equation}
Due to
$L_0=c_5\cdot\widetilde{\gamma}\delta_a$ and $r_0=s_0^{\lambda}$, we
can choose $s_0$ sufficiently small, so that when
\begin{displaymath}
r_0^{\alpha}\leq
\frac{\pi}{2}\gamma^{1-\frac{1}{\kappa\mu_0(\mu_0+1)}}~c_5^{-1}\, ,
\end{displaymath}
we
have $tL_0r_0^{\alpha+1}4^{-\lambda j(\alpha+1)}\leq
\frac{\pi}{2}t\gamma\cdot r_04^{-\lambda j}$. Thus,
\begin{equation}
|\widetilde{\chi}_{\nu}-\widetilde{\chi}_j|_{\widetilde{P}}\leq \frac{\pi}{2}t\gamma\cdot r_j~,
\end{equation}
so we can apply \Cref{Rlem:2} with
$\hat{\chi}=\widetilde{\chi}_{\nu}$ and get
\begin{eqnarray}
\mathcal {K}_{\gamma,t}^{(\nu+1)}&\supseteq&
\bigcap_{0<|k|\leq m_{\nu}}\mathcal
{H}_k(\widetilde{\chi}_{\nu})=\bigcap_{0<|k|\leq m_{\nu}}\Big\{x\in
\mathcal {K}_{\rho}~\big{|}~|\langle k,t\widetilde{\omega}^{(\nu)}
\rangle+2\pi l(t\widetilde{\omega}^{(\nu)},k)|\geq \frac{\pi}{2}\,
\frac{t\gamma}{|k|^{\tau}}\Big\}\nonumber\\[5pt]
&\supseteq&\mathcal {H}(\widetilde{\chi}_{\nu})\neq
\emptyset\,,\nonumber
\end{eqnarray}
Therefore, we get a contradiction to the maximality of the
considered chain.

Now (\ref{1}) allows the application of \Cref{Rlem:1} with
$\nu=\infty$ to the infinite chain $\langle \mathcal {L}_j\rangle_{0\leq
j<\infty}$ for $0\leq \mu\leq
\mu_0$. From (\ref{Cauchy}) we see that $(\widetilde{\chi}_j)$ is a
Cauchy sequence in $C^{\mu_0}(\widetilde{P},\mathbb{R}^{n+1})$,
hence
\begin{equation*}
\widetilde{\chi}_j\xrightarrow{j\rightarrow\infty}\widetilde{\chi}_{\infty}
\in C^{\mu_0}(\widetilde{P},\mathbb{R}^{n+1})\,.
\end{equation*}
As
$j\rightarrow\infty$ in (\ref{es:1}) and (\ref{es:2}) we have
\begin{equation}
|\widetilde{\chi}_{\nu}-\widetilde{\chi}_0|_{\widetilde{P}}^{\mu_0}\leq\frac{t\beta}{2}\quad
\mbox{and}\quad |\widetilde{\chi}_{\infty}|_{\widetilde{P}}\leq
C_{\cdot}+1.
\end{equation}
Applying \Cref{Rlem:4} again with
$\widetilde{\chi}=\widetilde{\chi}_{\infty}$\,,
$\widetilde{B}=\widetilde{P}$, $\theta=1$ and $\mathcal {K}=\mathcal
{K}_{\rho}$~, we obtain
\begin{equation}\label{es:3}
m(\mathcal {H}(\widetilde{\chi}_{\infty}))\geq m(\mathcal
{K}_{\rho})-\widetilde{M}\Theta^{\mu_0+1}\gamma^{\frac{1}{\mu_0}-\frac{1}{\kappa\mu_0}}\cdot
\beta^{-1-\frac{1}{\mu_0}}>0\,.
\end{equation}

It is clear that the links $\mathcal {L}_0,\cdots,\mathcal
{L}_{\nu}$ of our infinite chain $\langle \mathcal
{L}_j\rangle_{0\leq j\leq \infty}$ form a finite chain $\langle
\mathcal {L}_j\rangle_{0\leq j\leq \nu}$ with its extension $\langle
\widetilde{\mathcal {L}}_j\rangle_{0\leq j\leq \nu}$,
$\nu=0,1,\cdots.$ Applying again \Cref{Rlem:2}, but now with
$\widehat{\chi}=\widetilde{\chi}_{\infty}$, we obtain
\begin{equation*}
\mathcal {K}_{\gamma,t}^{(\nu)}\supseteq \mathcal {K}_{\gamma,t}^{(\nu+1)}
\supseteq \bigcap_{k\in \mathbb{Z}^n\atop 0<|k|\leq m_{\nu}}\mathcal
{H}_{k}(\widetilde{\chi}_{\infty})\supseteq\mathcal
{H}(\widetilde{\chi}_{\infty}),\quad \nu=0,1,\cdots.
\end{equation*}
From these relations and (\ref{es:3}), it implies that
\begin{equation}
\mathcal {K}_{\gamma,t}^{(\infty)}=\bigcap_{\nu=0}^{\infty}\mathcal {K}_{\gamma,t}^{(\nu)}
\supseteq \mathcal {H}(\widetilde{\chi}_{\infty})\neq \emptyset\,,
\end{equation}
and
\begin{equation}
m(\mathcal {K}_{\gamma,t}^{(\infty)})\geq m(\mathcal {K}_{\rho})-\widetilde{M}\Theta^{\mu_0+1}\gamma^{\frac{1}{\mu_0}-\frac{1}{\kappa\mu_0}}\cdot
\beta^{-1-\frac{1}{\mu_0}}.
\end{equation}
Denote
$\widetilde{M}:=\widetilde{M}\Theta^{\mu_0+1}$. That completes the proof of \Cref{thm:my-3}.
\end{proof}

Now we return to the proof of \Cref{clm:2}. Note
$L_0=c_5\cdot\gamma\delta_a$, where $\delta_a\leq\delta$. Hence for sufficiently small $\delta$ and $\gamma$, the conditions (\ref{1})
and (\ref{2}) can be satisfied. Combining
$B_{\gamma,t}^{(j)}\supseteq \mathcal {K}_{\gamma,t}^{(j)}$,
$j=0,1,\cdots$ and \Cref{thm:my-3}, we obtain
\begin{equation*}
B_{\gamma,t}^{(0)}\supseteq B_{\gamma,t}^{(1)}\supseteq\cdots\supseteq B_{\gamma,t}^{(\infty)}
:=\bigcap_{j=0}^{\infty}B_{\gamma,t}^{(j)}\neq \emptyset\,,
\end{equation*}
 and
\begin{equation*}
m(B_{\gamma,t}^{(\infty)})\geq m(\mathcal {K}_{\rho})-\widetilde{M}\gamma^{\frac{1}{\mu_0}-\frac{1}{\kappa\mu_0}}\cdot
\beta^{-1-\frac{1}{\mu_0}}.
\end{equation*}

\subsection{Permanent near-conservation of energy}
\label{subsec:energy}

Although it has been proved that symplectic algorithms cannot exactly preserve energy for general Hamiltonian systems (Ge-Marsden theorem~\cite{Ge}), it still has a good performance for the near-conservation of energy. By means of the backward error analysis, exponentially long time near-conservation of the energy can be
derived~\cite{Haire}. Using the KAM theory of symplectic algorithms, we can obtain the perpetual near-preservation of the energy on a Cantor set. In fact, the Hamiltonian function of any Hamiltonian system must be a first integral of the system, so the following corollary is achieved trivially from the conclusions of \Cref{cor:1}. Note that here and henceforth, we will use $h$ as the time step size to replace $t$.

\begin{corollary}\label{cor:2}
Consider a weakly non-degenerate integrable Hamiltonian system like \eqref{eq:basic} with an analytic Hamiltonian function $H:D\rightarrow \mathbb{R}$ $(where ~D\subset\mathbb{R}^{2n})$, and apply any $s$ order symplectic integrator with step size $h$ to this system, then there exists $h_0$, such that when $h<h_0$ and the initial value $(x_0,y_0)$ is located in some Cantor set $D_{\gamma,h}$ of the phase space $D$, the perpetual near-preservation of $H$ can be achieved under the symplectic integrator, {\rm i.e.},
\begin{equation}\label{H-error}
|H(x_n,y_n)-H(x_0,y_0)|\leq ch^s
\end{equation}
where the measure of $D_{\gamma,h}$ depends on step size $h$; $(x_n, y_n)=(x_0+nh, y_0+nh)$, $n=0,1,2,\cdots$;
the constant $c$ don't depend on $h$ and $n$.
\end{corollary}

From the viewpoint of backward error analysis or formal energy, when a symplectic integrator applies to an analytic weakly non-degenerate integrable Hamiltonian system, the numerical solution can be interpreted as the exact solution of a modified Hamiltonian system that is a formal series in powers of the step size~\cite{Haire}. Although The modified Hamiltonian function or formal energy, denoted by $\widetilde{H}$, is non-convergent in any region~\cite{Sanz-Serna91b,Wang,Tang}, the existence of the numerical invariant tori in phase space implies that it is well-defined on some Cantor set and close to the original energy $H$ up to the order $O(h^s)$. Using these results and the triangle inequality, we can get \eqref{H-error} again.

\section{Numerical experiments}
\label{sec:experiments}

In this section, we will investigate the preservation and destruction of invariant tori by numerical experiments. Consider an analytic weakly non-degenerate integrable Hamiltonian system with $n$ degrees of freedom like (\ref{eq:basic}). As explained before, the whole phase space of the system has a foliation
into $n$-dimensional invariant tori under the action-angle variables $(p,q)\in B\times T^n$, where $B$ is a connected bounded open subset of $\mathbb{R}^n$ and $T^n$ the standard $n$-dimensional torus. The motion on each torus $\{p=p^*\}$ is a linear flow determined by the equation
\begin{equation}\label{condition-4}
\dot{q}=\omega(p^*):=\frac{\partial H(p^*)}{\partial p}\,,
\end{equation}
where the Hamiltonian function $H$ only depends on the action variables $p$, and $\omega: B\rightarrow \mathbb{R}^n$ is the frequency map of the system.

From \Cref{thm:my-2}, we can see that applying a symplectic integrator of order $s$ to this system, if the step size $h$ is sufficiently small, the symplectic difference scheme also has invariant tori (named numerical invariant tori as before) forming a Cantor set of the phase space. Specifically, there exist a Cantor set $\mathcal {K}_{\gamma,h}\subset B$ and a mapping $\omega_{\gamma,h}:\mathcal{K}_{\gamma,h}\rightarrow\Omega_{\gamma,h}$ such that the one-step map $S_h$ of the symplectic scheme conjugate to a rotation: $(p,q)\rightarrow (p,q+h\omega_{\gamma,h}(p))$, when it is restricted to $\mathcal{K}_{\gamma,h}\times T^n$. Here $\Omega_{\gamma,h}$ is defined in (\ref{condition-1}), and a sufficiently small $\gamma>0$ can be given beforehand, so we can write $\omega_{\gamma,h}=\omega_h$ for simplicity. The frequency map $\omega_h$, defined in $\mathcal{K}_{\gamma,h}$, can be interpreted as the "frequency" of the numerical solution on these numerical invariant tori, in the sense that there exists a modified Hamiltonian system so that its time-$h$ flow exactly interpolates the numerical solution on these tori. Generally, $\omega_h$ depends on the step size $h$, and its image is not equal to the one of $\omega$ when restricted to the set $\mathcal{K}_{\gamma,h}$, but it is an $O(h^s)$-approximation of $\omega$ by \Cref{rm:3}.

Note that the values of $\omega_h$ are in $\Omega_{\gamma,h}$, that means $h$ and $\omega_h$ must satisfy the Diophantine condition (see (\ref{condition-1}))
\begin{equation}\label{condition-2}
|e^{i\langle
k,h\mathbf{\omega}_h\rangle}-1|\geq
\frac{h\gamma}{|k|^{\tau}},~\forall~ k\in\mathbb{Z}^n\backslash
\{0\}\,,
\end{equation}
otherwise resonance may be occur. Following Shang~\cite{Shang00}, we define the resonant step size if the relationship
\begin{equation}\label{condition-3}
h=\frac{2\pi l}{\langle k,\omega_h\rangle}
\end{equation}
holds for some $k\in\mathbb{Z}^n\backslash\{0\}$ and $l\in \mathbb{Z}$. The resonances introduced by numerical discretization is referred as numerical resonances~\cite{Moan}. The number $|k|=k_1+\ldots+k_n$ is called resonance order. In contrast, the step sizes satisfying (\ref{condition-2}) will be called non-resonant step sizes (or Diophantine step sizes~\cite{Shang00}).

The frequency $\omega_h$ is difficult to predict, so the resonant step sizes cannot be determined in advance by using the relationship (\ref{condition-3}). In the following, we will relate the resonant step size to the variation of the energy error, then verify our theoretical results for one degree and multi-degrees of freedom systems, respectively.

\subsection{One degree of freedom system}
\label{subsec:ex-1}

We take a typical Hamiltonian system, simple pendulum, as an example to study the preservation and destruction of invariant tori. If the units are chosen in such a way that the mass of the blob, the length of the rod and the acceleration of gravity are all unity, then the equation of motion is $\ddot{x}=-\sin{x}$, where $x$ is the angle of the pendulum deviation from the vertical. By introducing canonical variables $(p,q)=(\dot{x},x)$, this equation can be rewritten in the Hamiltonian form
\begin{equation} \label{eq:3}
\left\{ \begin{aligned}
         \dot{p} &= -\frac{\partial H}{\partial q}(p,q) \\
                  \dot{q}&=\frac{\partial H}{\partial p}(p,q)
                          \end{aligned} \right.
\end{equation}
with Hamiltonian $H(p,q)=p^2/2+(1-\cos{q})$. The period of the pendulum is
\begin{equation} \label{eq:4}
T_0=2\sqrt{2}\int_0^{q_m}\frac{\mathrm{d}q}{\sqrt{\cos{q}-\cos{q_m}}}\,,
\end{equation}
where $q_m$ is the largest deflection angle of simple pendulum. Due to the nonlinear dependence of the period on the amplitude, it is easy to see that the system satisfies Kolmogorov non-degeneracy condition, and thereby it also satisfies R\"{u}ssmann non-degeneracy condition.

Here we apply the implicit midpoint (IM) scheme to the system (\ref{eq:3}). This is a symplectic algorithm of order 2. Starting from the initial generalized coordinate $q_0=0$ with the conjugated generalized momenta $p_0=0.7$, the phase flow of this system forms a closed curve as the invariant torus of the system. The corresponding period can be computed by (\ref{eq:4}), $T_0\approx 6.4901$, and the (angular) frequency of the motion is $\omega=\frac{2\pi}{T_0}\approx 0.9681$. It is of interest to study how one choose the step size $h$ not to introduce instabilities for a priori stable orbits.

From the preceding results, the invariant torus of the system can be preserved by symplectic integrators as long as the time-step $h$ is small enough. Here we illustrate the existence of the numerical invariant torus by the frequency spectrum analysis. First, we integrated the system numerically with $h=0.01$ and initial condition $(p_0,q_0)=(0.7,0)$, and recorded the values of $q_n$ at $n=0,1,\ldots,10^5-1$. This yielded a time series consisting of $N=10^5$ numbers. Then we used NAFF (Numerical Analysis of Fundamental Frequencies) algorithm, proposed by Laskar~\cite{Laskar88}, to compute the frequencies and amplitudes of the numerical solution that is illustrated in \cref{danbai:4th.1}. The advantage of Laskar's method is that it recovers the fundamental frequencies with an error that falls off as $T^{-4}$~\cite{Laskar99} over a finite time span $[-T,T]$, compared with $T^{-1}$ for the ordinary FFT method.

\begin{figure}[h!]
\centering
\subfloat{\label{danbai:4th.1}\includegraphics[scale=0.4]{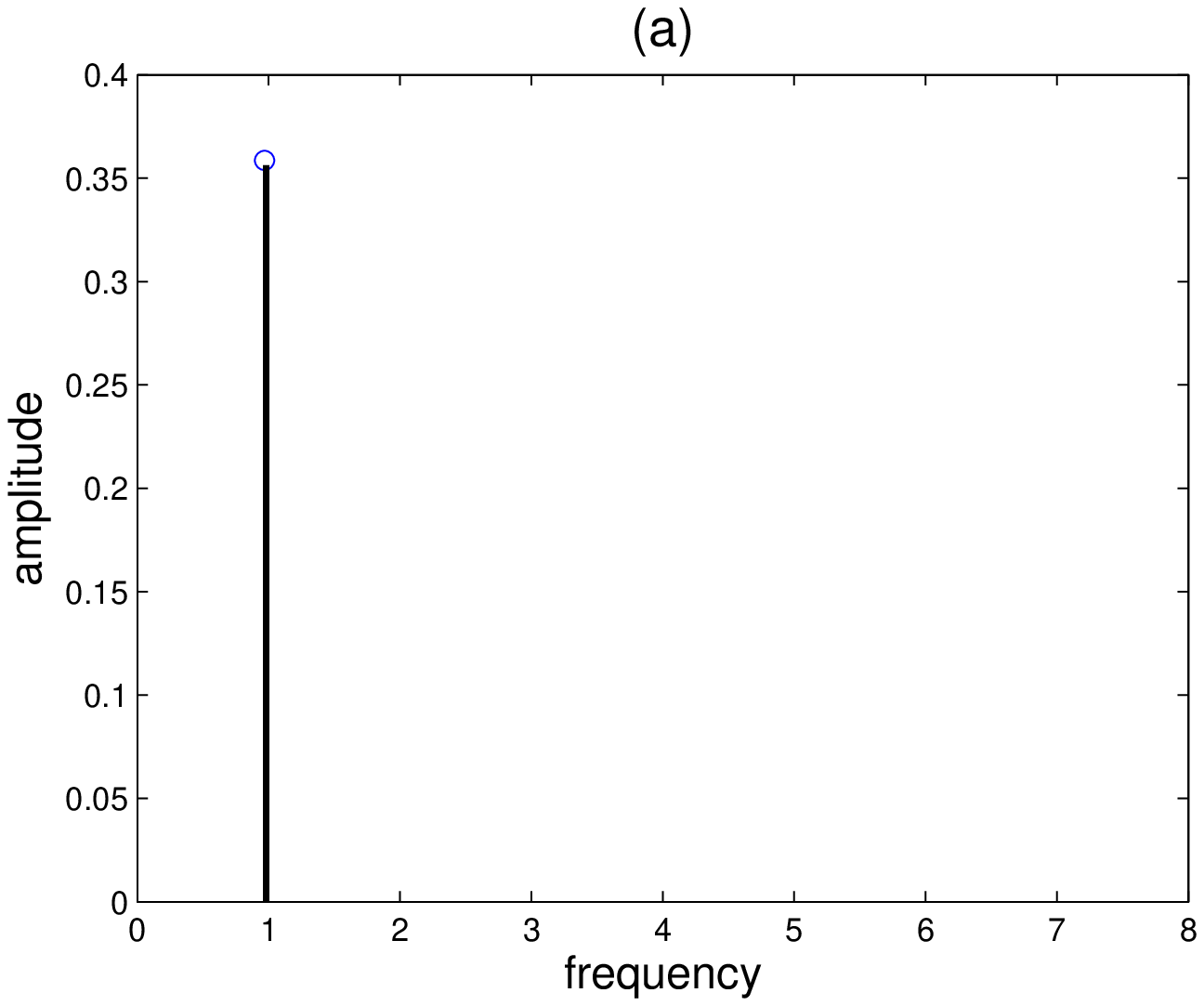}}\qquad
\subfloat{\label{danbai:4th.2}\includegraphics[scale=0.4]{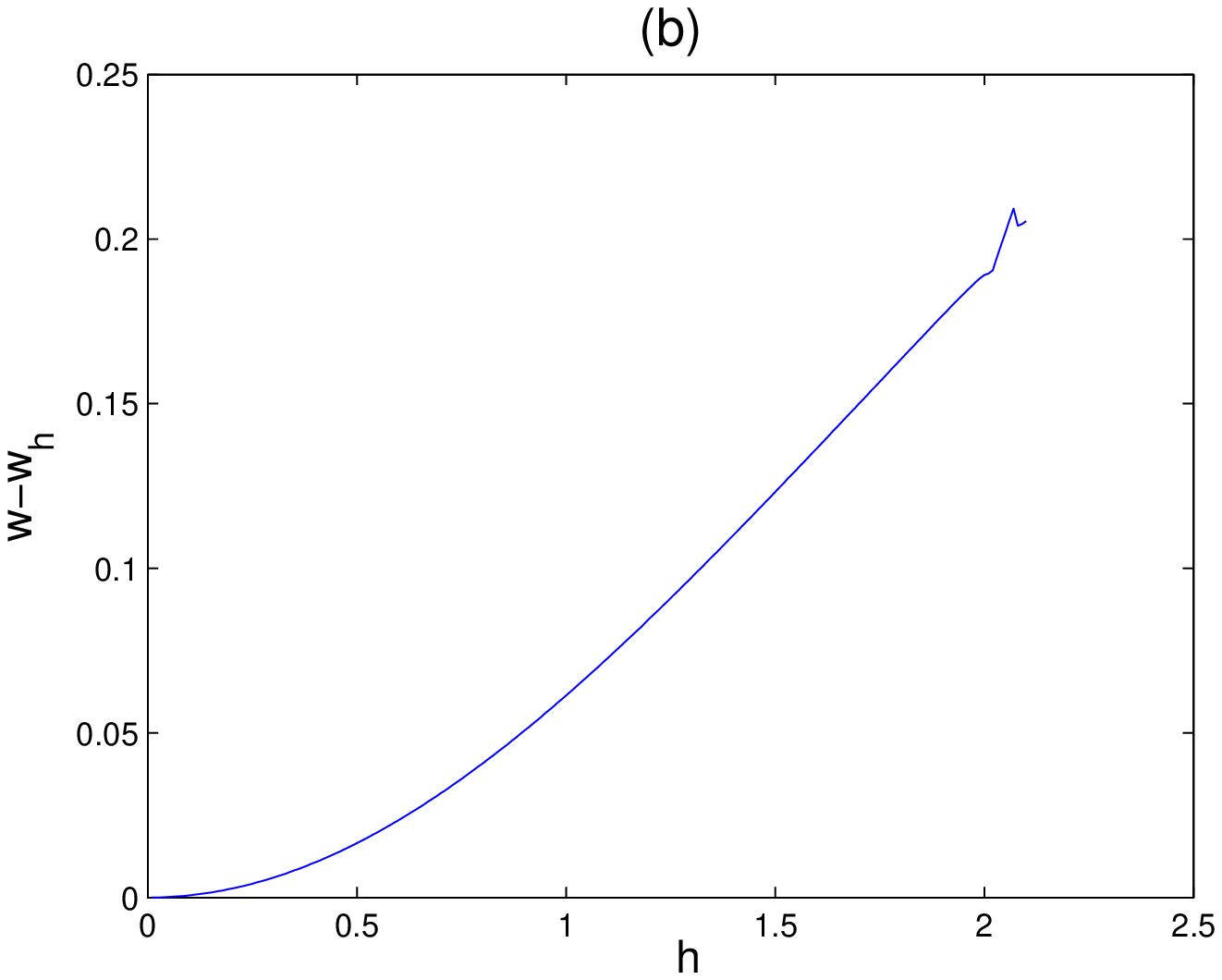}}
\caption{{\rm (a)} Spectrum of the IM scheme applied to simple pendulum. {\rm (b)} The difference between $\omega_h$ and $\omega$ as a function of step size $h$.}
\label{danbai:4th}
\end{figure}
\begin{figure}[h!]
\centering
\subfloat{\includegraphics[scale=0.4]{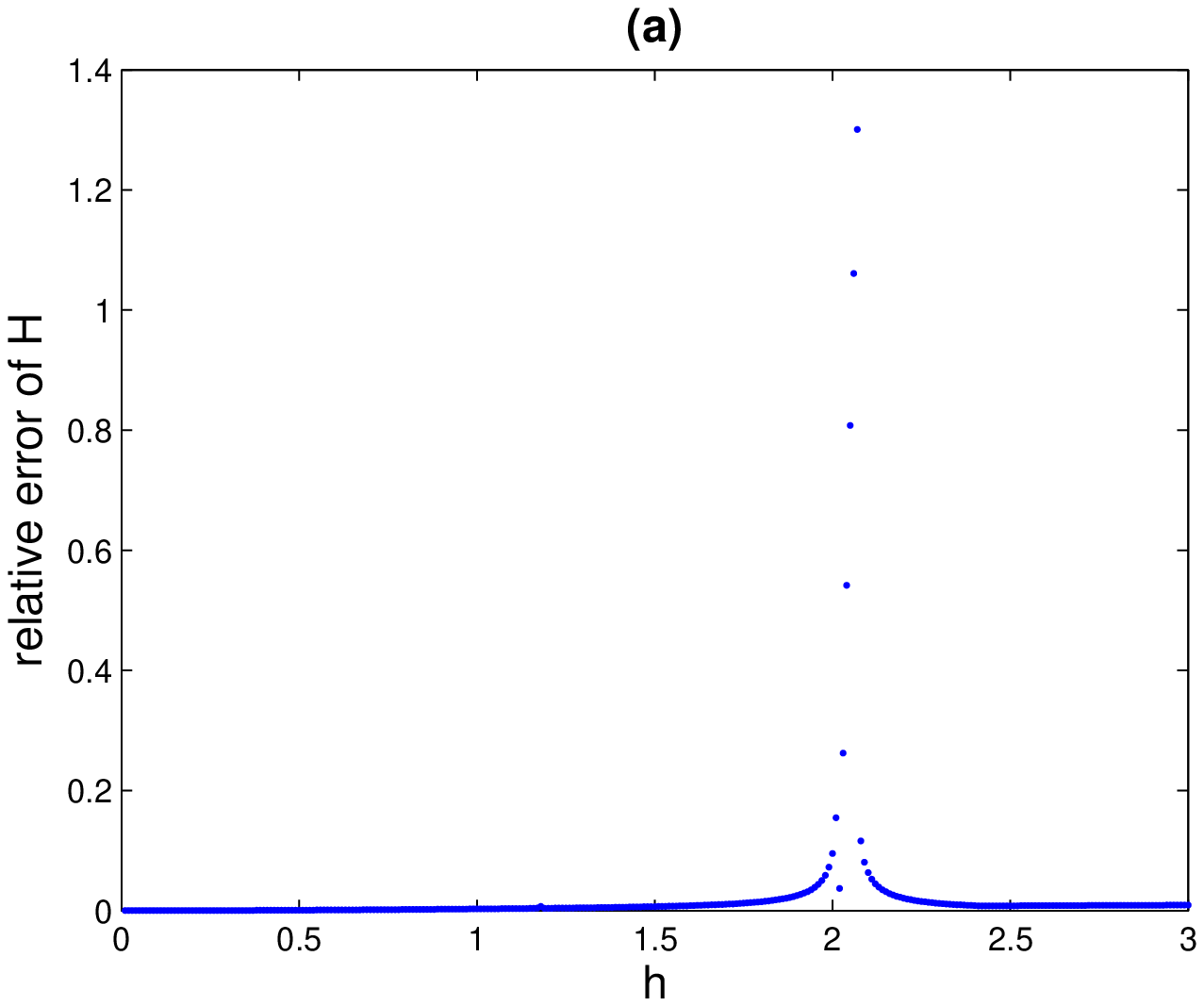}}\qquad
\subfloat{\includegraphics[scale=0.4]{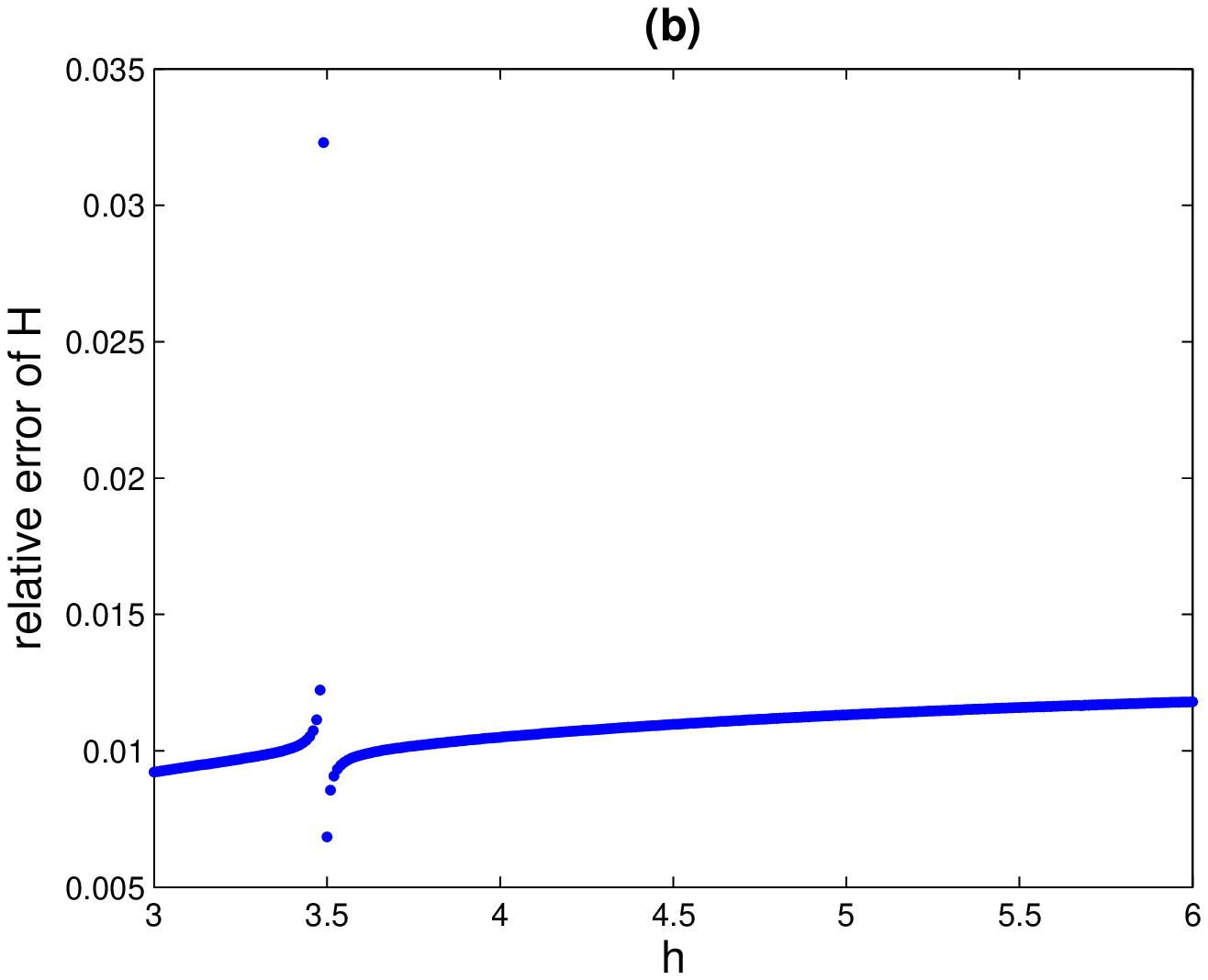}}
\caption{Relative numerical errors of the energy $H$ in the infinity norm as a function of the step size $h$ over $10^5$ steps (a) The range of $h$ is from 0.01 to 3. (b) The range of $h$ is from 3 to 6.}
\label{danbai:1nd}
\end{figure}

In \cref{danbai:4th.1}, there is one spectral line at the frequency $\omega=0.9681$. That is consistent with the periodicity of the numerical solutions. In addition,  the errors between $\omega_h$ and $\omega$ with increasing $h$ are plotted in \cref{danbai:4th.2}. It can be observed that there exists $h_0>0$ (say $h_0=1$) such that when $0<h<h_0$, the frequency errors are of the order $O(h^2)$ (see \Cref{rm:3}).

From \Cref{cor:1}, we know that symplectic integrators approximately conserve the values of first integrals of the system with the accuracy of $h^2$ when the time step $h$ is non-resonant (see (\ref{condition-2})). In contrast, if $h$ is resonant, the invariant torus of the system will break in general, which leads to a sudden increase in the error of the first integrals. Therefore, we can identify the resonant step sizes by examining the error of first integrals with the change of the step size. For the pendulum system, we illustrate the relative errors of the energy $H$ as a function of the step size $h$ under the IM scheme in \cref{danbai:1nd}. It is observed that there are two peaks in the errors, corresponding to the step sizes $h\approx 2.05$ and $h\approx 3.5$ respectively, which implies that the numerical resonance occurs at these two step sizes.

\begin{figure}[h!]
\centering
\subfloat
{\includegraphics[scale=0.52]{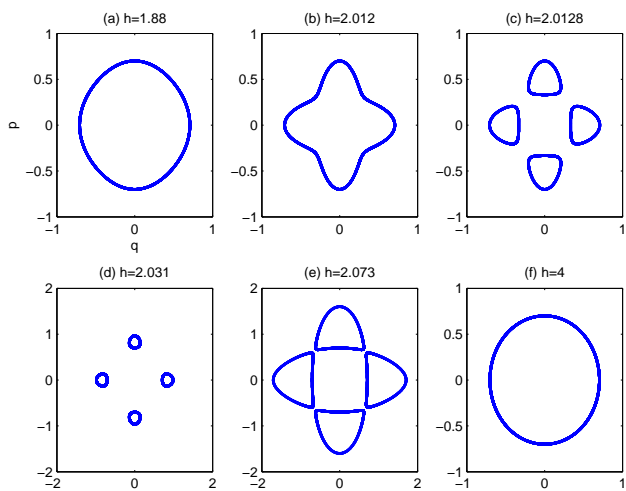}}
 \caption{Phase diagrams of numerical solutions for increasing step sizes near $h=2.1$.}
\label{danbai:2nd}
\end{figure}

In \cref{danbai:2nd,danbai:3nd}, we display the variation of phase diagrams of the numerical solutions with the step sizes, in the vicinity of $h=2.05$ and $h=3.5$ respectively. For $h=2.05$, the corresponding frequency $\omega_h=0.7627$ computed by NAFF algorithm, such that the equation (\ref{condition-3}) is satisfied approximately for $k=4$ and $l=1$. That means the fourth-order resonance occurs near the step size $h=2.05$, and correspond to the emergence of four separate islands in \cref{danbai:2nd}. Similarly, for $h=3.5$ the frequency $\omega_h=0.5990$. Thus the equation \eqref{condition-3} is satisfied approximately for $k=3$ and $l=1$. This corresponds to the third-order resonance in \cref{danbai:3nd} near $h=3.5$. Notice that when numerical resonances occur, the numerical solutions cannot be viewed as exact ones of a modified Hamiltonian system close to the original one, since the numerical solutions would not lie on closed smooth curves.

Finally, we point out that the method is only suitable to identify those apparent numerical resonance phenomena. In theory, there are infinite step sizes to make the relationship (\ref{condition-3}) hold, but most of the destruction of invariant tori caused by resonant step sizes are very slight, in particular, when the step size is small.

\begin{figure}[h!]
\centering
\subfloat
{\includegraphics[width=11.5cm,height=3.5cm]{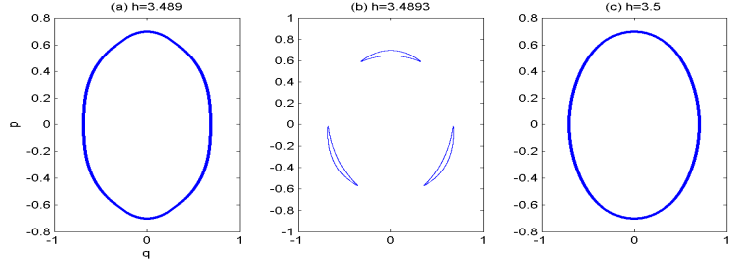}}
 \caption{Phase diagrams of numerical solutions for increasing step sizes near $h=3.5$.}
\label{danbai:3nd}
\end{figure}

\subsection{Multi-degrees of freedom system}
\label{subsec:ex-2}

In this section, we consider an integrable Hamiltonian system with multiple degrees of freedom, which satisfies R\"{u}ssmann non-degeneracy condition but not Kolmogorov one. The example is from R\"{u}ssmann~\cite{Ru1989}, and the Hamiltonian is
\begin{equation}\label{Hamilaton}
K(x,y)=\Big(\frac{x_1^2+y_1^2}{2}\Big)^2\Big(1+\frac{x_2^2+y_2^2}{2}\Big)+\Big(\frac{x_1^2+y_1^2}{2}\Big)^3\frac{x_3^2+y_3^2}{2}\,.
\end{equation}
By the symplectic coordinate transformation $\Psi$:
\begin{equation}\label{Hamilaton-1}
\left\{ \begin{aligned}
         x_i&=\sqrt{2p_i}\cos{q_i},\\
     y_i&=\sqrt{2p_i}\sin{q_i},\quad (i=1,2,3).
                          \end{aligned} \right.
                          \end{equation}
the Hamiltonian becomes $H(p)=K\circ\Psi(p,q)=p_1^2+p_2p_1+p_3p_1^3$, and the system takes the simple form
\begin{equation}\label{Hamilaton-2}
\left\{ \begin{aligned}
         \dot{p}&=0,\\
     \dot{q}&=\omega(p)=\frac{\partial H}{\partial p}(p).
                          \end{aligned} \right.
                          \end{equation}
It is easy to verify that $\omega$ satisfies the weakly non-degeneracy condition but not the non-degeneracy one.

We apply four numerical schemes to the system (\ref{Hamilaton}) for comparison. They are implicit midpoint (IM), St\"{o}rmer-Verlet, symplectic Euler and Runge scheme (refer to~\cite{Haire}), respectively. Both IM and St\"{o}rmer-Verlet scheme are symplectic algorithms of order 2. Symplectic Euler method is a first order symplectic method, while Runge scheme is a 2nd order non-symplectic.

Starting from the initial values $\mathbf{x}_0=[0.2, 0.1, 0.4\sqrt{2}]^{\mathrm{T}}$ and $\mathbf{y}_0=[0.37, 0.2, 0.53]^{\mathrm{T}}$, the solution of the system is a quasiperiodic motion on some invariant torus in phase space by (\ref{Hamilaton-2}), and the corresponding frequency vector $\mathbf{\omega}=[0.1884, 0.0078, 6.9198\times 10^{-4}]^{\mathrm{T}}$. In order to verify the existence of numerical invariant torus when $h$ is small, we apply the IM scheme to integrate the system with $h=0.01$ over $10^5$ steps. The calculated frequency vector $\omega_h$ is almost the same as $\omega$, which indicate the quasiperiodic character of the numerical solutions.

There are three independent first integrals (or invariants) in this system, i.e.,
\begin{equation}
I_i=\frac{x_i^2+y_i^2}{2}, \qquad i=1,2,3.
\end{equation}
The energy $K$ is an assemble of them. In \cref{Russman:2nd}, we show the relative errors of the three first integrals and the energy $K$ as a function of the step size $h$ for different numerical schemes, and we can identify the resonant step sizes by examining the variation of the errors. Due to relatively high accuracy of the IM scheme, one can distinguish more numerical resonances in \cref{Russman:2nd.1}. For example, there is a peak at about $h=0.14$, that implies the step size is a resonant one with the corresponding $k=(231,64,12)\in\mathbb{Z}^3$ and $l=1$ in (\ref{condition-3}). In addition, it is observed from \cref{Russman:2nd.1} that the resonance steps have approximately equal intervals with the length $0.14$.

In \cref{Russman:2nd.2,Russman:2nd.3}, we are only able to identify relatively few resonance steps, due to the low accuracy of the algorithms. The corresponding values of $k$ are marked in the graph. As a comparison, \cref{Russman:2nd.4} shows the relative errors of the first integrals under the 2nd order non-symplectic Runge scheme, from which no obvious peaks are observed. For multi-degrees of freedom Hamiltonian system, we cannot directly observe the destruction of invariant torus through phase diagram. However, one can investigate the occurrence of numerical resonances indirectly by the above method.

\begin{figure}[t]
\centering
\subfloat
{\label{Russman:2nd.1}\includegraphics[width=6.0cm]{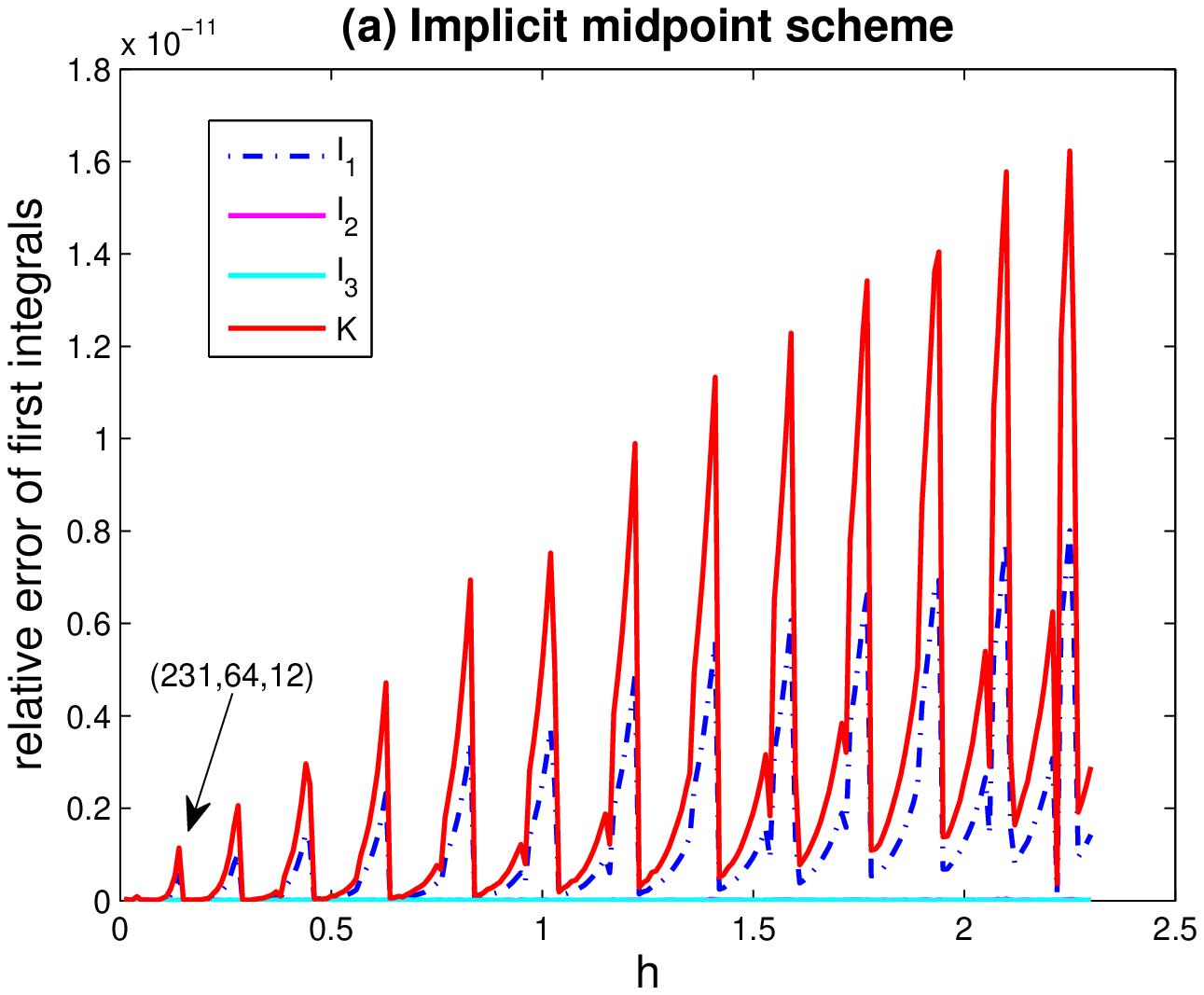}}\qquad
\subfloat
{\label{Russman:2nd.2}\includegraphics[width=6.0cm]{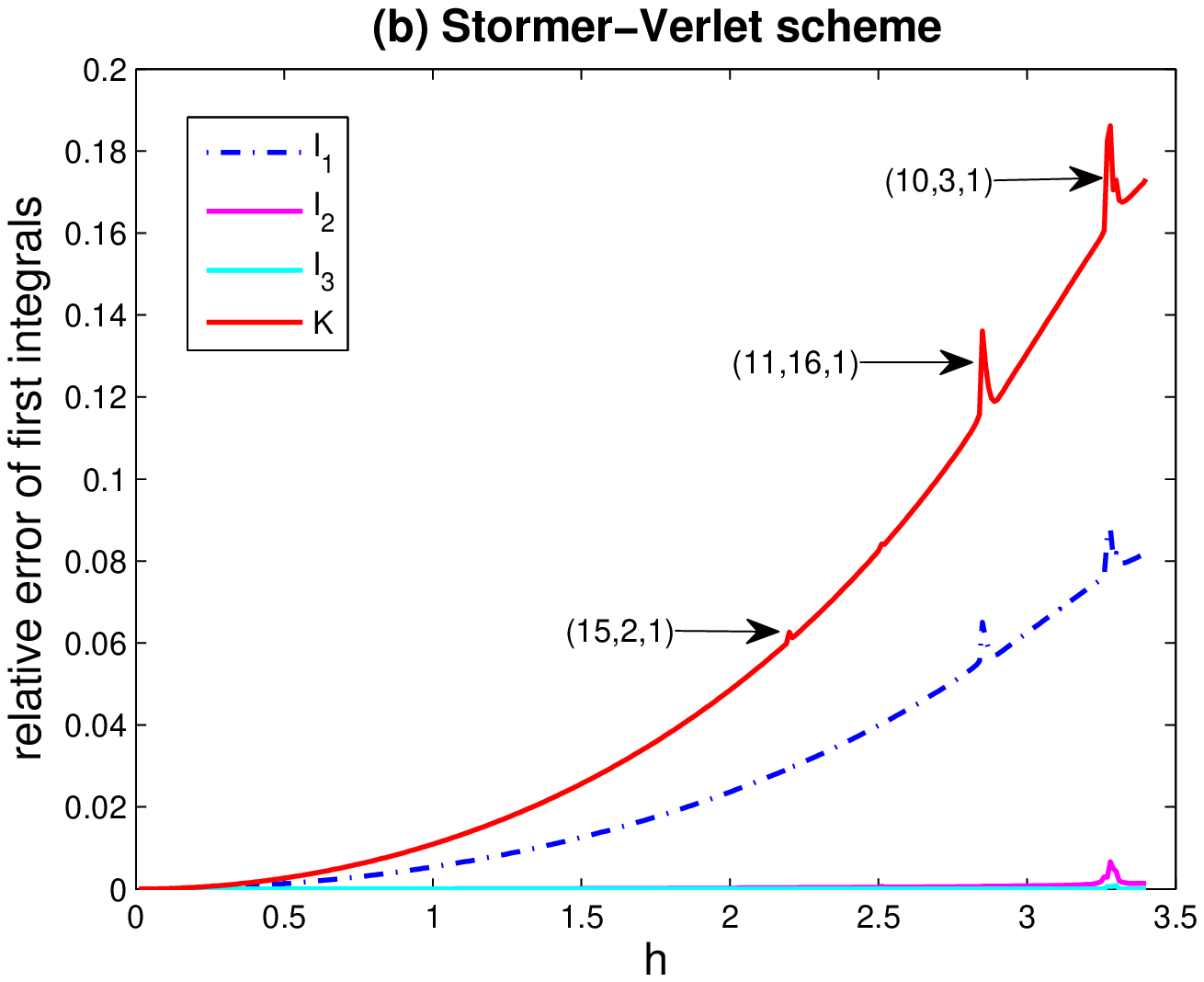}}\\
\subfloat
{\label{Russman:2nd.3}\includegraphics[width=6.0cm]{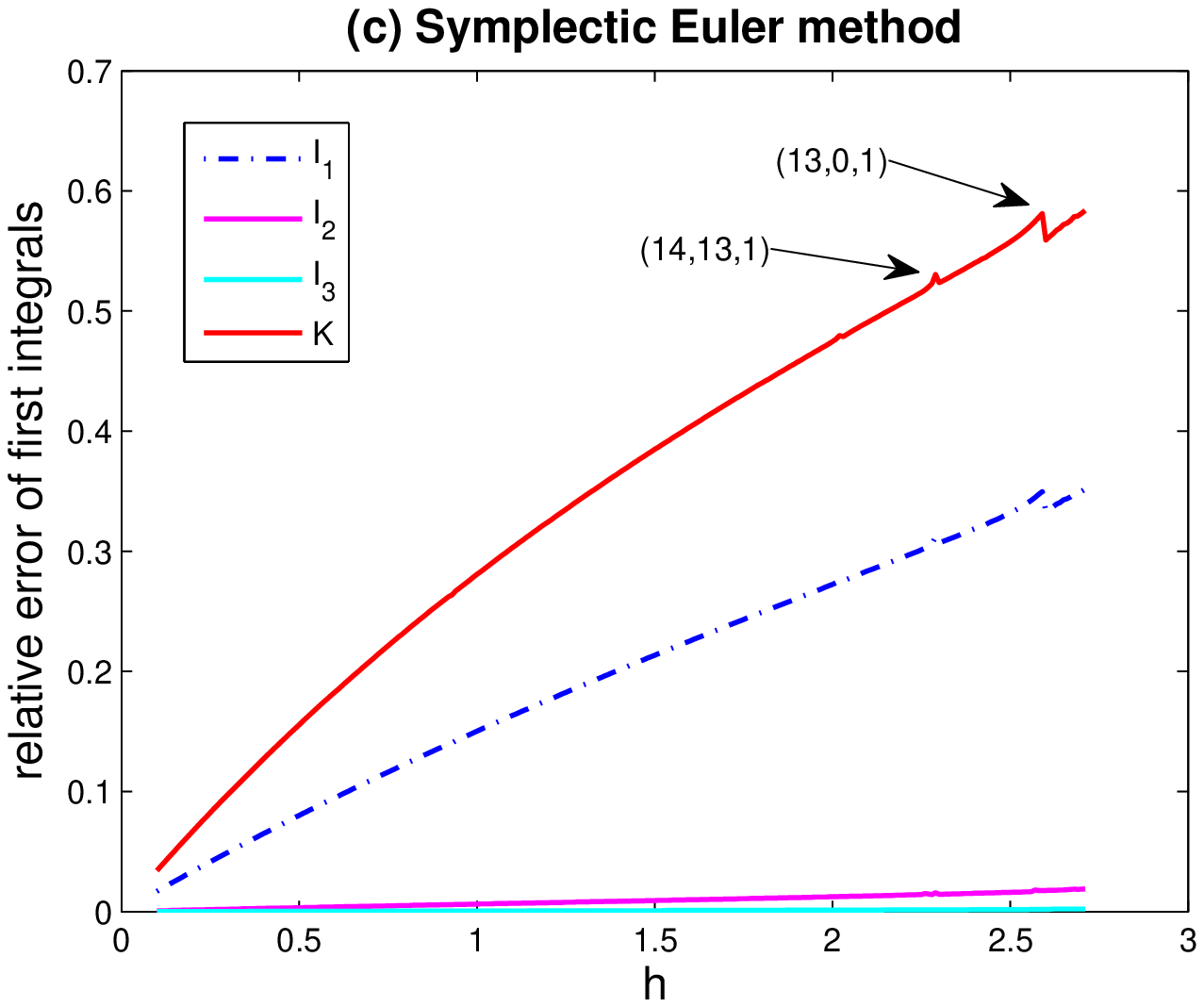}}\qquad
\subfloat
{\label{Russman:2nd.4}\includegraphics[width=6.0cm]{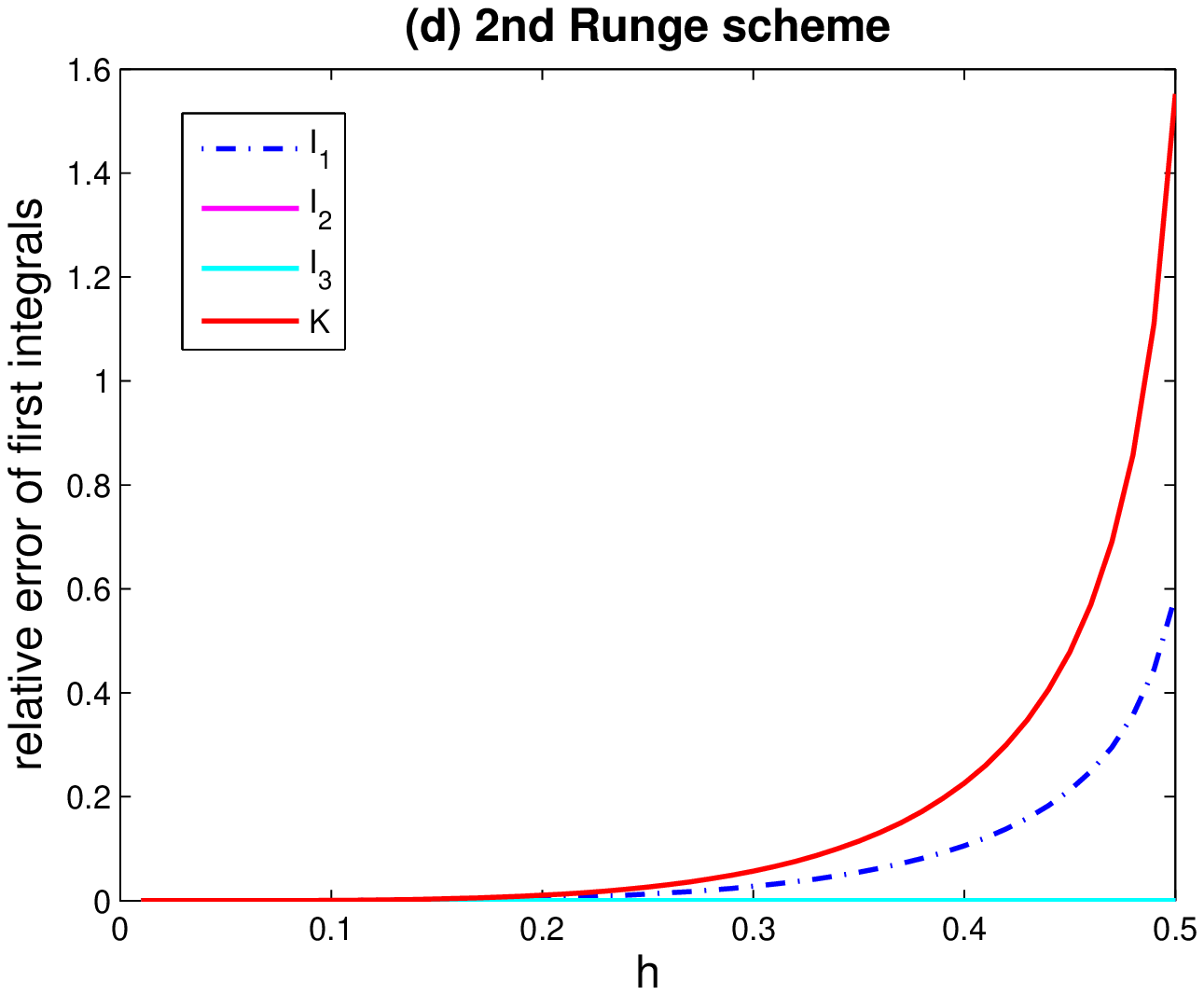}}
\caption{Relative errors of $I_1$, $I_2$, $I_3$ and $K$ in the infinity norm as a function of the
step size $h$ over $10^5$ steps for four different numerical schemes.}
\label{Russman:2nd}
\end{figure}

\section{Discussion}
\label{sec:conclusions}

In the present work, we generalize Shang's results (1999, 2000) on the existence of numerical invariant tori for symplectic integrators. To be specific, we prove that when the non-degeneracy condition is weakened from Kolmogorov's one to R\"{u}ssmann's one, most non-resonant invariant tori of the integrable system still can be preserved by symplectic integrators if the step size $h$ is sufficiently small. This type of theorem helps to understand the qualitative behavior of symplectic integrators. In particular, our result contributes to the nonlinear stability analysis of symplectic integrators for a more general class of integrable Hamiltonian systems.

More quantitative results can be achieved, such as the estimate of the error on the frequencies between the numerical integrators and the exact ones, near-preservation of first integrals and so on. Furthermore, relating the resonant step sizes to the variation of the error of first integrals, we can investigate the preservation and destruction of invariant tori under symplectic integrators by numerical experiments, thereby verifying our theoretical results.

Thanks to the numerical KAM theory, the permanent numerical stability of symplectic integrators can be obtained on those preserved numerical invariant tori, which form a relatively large measure Cantor set in phase space for small step size. However, when the resonance occurs, the corresponding invariant torus will be broken. So the permanent stability of numerical solutions appears to be impossible in this case, especially for high dimensional systems. Fortunately, an exponentially long time stability, in a negative power of the step size, can be derived by a Nekhoroshev-like theorem~\cite{Ding} (see also~\cite{Benettin,Lubich,Stoffer1998} where some similar results are available on the exponential stability of symplectic algorithms in a neighborhood of invariant tori.). This kind of theorems can greatly ease the occurrence of the instability of symplectic integrators in practice. Combining these two types of theorem (i.e., KAM-like theorem and Nekhoroshev-like one), one can get a more complete characterization of the qualitative behavior of symplectic integrators.


\Acknowledgements{This work was supported by National Natural Science Foundation of China (Grant No. 11671392).}




\begin{appendix}
\section{\label{sec:compint}}
In this appendix, we provide the proofs of \Cref{Rlem:2,Rlem:4}.

\subsection{Proof of \Cref{Rlem:2}}
\label{sec:proof:1}

\begin{proof} From (\ref{chi}) with $\mu=0$, (\ref{L_0}) and
c(n,\,0)=1, we infer that
$$|\widetilde{\chi}_j|_{\widetilde{P}}\leq C_{\cdot}+L_0r_0^{\alpha+1}
\frac{(\frac{1}{4})^{\lambda(\alpha+1)}}{1-(\frac{1}{4})^{\lambda(\alpha+1)}}\leq
C_{\cdot}+1\,.$$ $\forall~ x\in \mathcal
{H}_k(\widehat{\chi})$, i.e. $|\langle
k,t\widehat{\omega}(x)\rangle+2\pi l(t\widehat{\omega},k) |\geq\pi
\frac{t\gamma}{|k|^{\tau}}$ for $k\neq 0$, then
\begin{eqnarray*}
|\langle k,t\widetilde{\omega}^{(j)} \rangle+2\pi
l(t\widetilde{\omega}^{(j)},k)|&\geq&
|\langle k,t\widehat{\omega} \rangle+2\pi l(t\widetilde{\omega}^{(j)},k)|-|\langle k,t\widetilde{\omega}^{(j)}-t\widehat{\omega} \rangle|\\
 &\geq& |\langle k,t\widehat{\omega} \rangle+2\pi
 l(t\widehat{\omega},k)|-|k|_2|t\widetilde{\omega}^{(j)}-t\widehat{\omega}|_{\widetilde{P}}\\
 &\geq&\pi
 \frac{t\gamma}{|k|^{\tau}}-m_j|t\widetilde{\omega}^{(j)}-t\widehat{\omega}|_{\widetilde{P}}\,,
 \end{eqnarray*}
for $|k|\leq m_j$\,. Since $m_j^{\tau+1}=\frac{1}{r_j}$ and
$|t\widetilde{\omega}^{(j)}-t\widehat{\omega}|_{\widetilde{P}}\leq
\frac{\pi}{2}t\gamma\, r_j$\,, we have
$$|\langle
k,t\widetilde{\omega}^{(j)} \rangle+2\pi
l(t\widetilde{\omega}^{(j)},k)|\geq \frac{\pi}{2}
\frac{t\gamma}{m_j^{\tau}}\,.$$ Therefore,
\begin{equation}\label{H-k}
\mathcal {H}_k(\widehat{\chi})\subseteq\mathcal {H}_{kj}:=\Big\{x\in
\mathcal {K}_{\rho}~\big{|}~|\langle k,t\widetilde{\omega}^{(j)}
\rangle+2\pi l(t\widetilde{\omega}^{(j)},k)|\geq \frac{\pi}{2}\,
\frac{t\gamma}{m_j^{\tau}}\Big\}\,.
\end{equation}
By the definition \eqref{tilde-k}, we deduce
\begin{equation}
\mathcal {K}_{\gamma,t}^{(j+1)}=\mathcal {K}_{\gamma,t}^{(j)}\bigcap \Big(\bigcap\limits_{|k|\leq
m_j} \mathcal {H}_{kj}\Big)\supseteq\mathcal
{K}_{\gamma,t}^{(j)}\bigcap \Big(\bigcap\limits_{|k|\leq m_j}
\mathcal {H}_{k}(\widehat{\chi})\Big)\,,
\end{equation}
for $0\leq j\leq \nu$.
Therefore, with the help of induction, we obtain
\begin{equation}
\mathcal {K}_{\gamma,t}^{(\nu+1)}\supseteq \mathcal
{K}_{\gamma,t}^{(0)}\bigcap \Big(\bigcap\limits_{|k|\leq m_{\nu}}
\mathcal {H}_{k}(\widehat{\chi})\Big)=\bigcap\limits_{|k|\leq
m_{\nu}} \mathcal {H}_{k}(\widehat{\chi}).
\end{equation}
The monotonicity, i.e., $\mathcal
{K}_{\gamma,t}^{(j+1)}\subseteq \mathcal {K}_{\gamma,t}^{(j)}$
for $0\leq j\leq\nu$, is a trivial conclusion. Thus, this lemma is proved.

\end{proof}

\subsection{Proof of \Cref{Rlem:4}}
\label{sec:proof:2}

\begin{proof} Denote
\begin{equation}
F_k(x)=\frac{\langle k,t\widetilde{\omega}(x)\rangle }{2}, \quad
t\in [0,1]~.
\end{equation}
Because $F_k$ is a continuous function on the compact set $\mathcal
{K}$, there exist a finite number of
sets, say $\Big\{[(i-1)\pi,i\pi]\Big\}_{i=-\varphi_k}^{\varphi_k}$,
$\varphi_k\in \mathbb{Z}^{+}$, to cover the range of
$F_k$ on $\mathcal {K}$. Define
\begin{equation}
A_k^i=F_k^{-1}\Big([(i-1)\pi,i\pi]\Big)\,,~ \mbox{for}~
i=-\varphi_k,-\varphi_k+1,\cdots,\varphi_k.
\end{equation}

Choose $l=l(t\widetilde{\omega},k)\in
\mathbb{Z}$ such that
\begin{equation}\label{belong}
\Big|\frac{\langle k,t\widetilde{\omega}\rangle+2\pi l}{2}\Big|\in
\Big[0,\frac{\pi}{2}\Big]\,,
\end{equation}
must also be finite for fixed $k$, and when $x\in A_k^i$,
$l(t\widetilde{\omega},k)$ depends only on $k$. It means that $l$
will take $2\varphi_k+1$ constant values, say
$\{l_i\}_{i=1}^{2\varphi_k+1}$, on the $2\varphi_k+1$ domains
$\{A_k^i\}_{i=-\varphi_k}^{\varphi_k}$ for fixed $k$. Using
(\ref{belong}), one has $|2\pi l|\leq |\langle
k,t\widetilde{\omega}\rangle|+\pi\leq
|k||\widetilde{\chi}|_{\widetilde{B}}\leq M_0 |k|$, so
\begin{equation}\label{|l|}
|l|\leq\frac{M_0}{2\pi}|k|\quad\mbox{and}\quad 2\varphi_k+1\leq
\frac{M_0}{\pi}|k|+1\leq (\frac{M_0}{\pi}+1)|k|.
\end{equation}
for $k\neq 0$. Denote
\begin{equation}\label{[chi]}
[\chi]_k(x)=\langle c,\chi(x)\rangle\quad \mbox{and}
\quad[\widetilde{\chi}]_k(x)=\langle c,\widetilde{\chi}(x)\rangle\,,
\end{equation}
where
\begin{equation*}
c=\frac{\tilde{k}}{|\tilde{k}|},\quad
\tilde{k}=(k,l(t\widetilde{\omega},k)).
\end{equation*}
Note $l$ is constant on $A_k^i$, so
\begin{eqnarray}\label{|chi|-1}
\Big|[\chi]_k-[\widetilde{\chi}]_k\Big|_{A_k^i}^{\mu_0}&:=&\max_{0\leq\nu\leq\mu_0
\atop x\in A_k^i}\max_{a\in
\mathbb{C}^n\atop |a|_2=1}\big|D^{\nu}([\chi]_k-[\widetilde{\chi}]_k)(x)(a^{\nu})\big|_2\nonumber\\
&=&\max_{0\leq\nu\leq\mu_0\atop x\in A_k^i}\max_{a\in
\mathbb{C}^n\atop|a|_2=1}\frac{1}{|\tilde{k}|}|\langle
\tilde{k},D^{\nu}(\chi-\widetilde{\chi})(x)(a^{\nu})\rangle|_2\nonumber\\
&\leq&\max_{0\leq\nu\leq\mu_0\atop x\in A_k^i}\max_{a\in
\mathbb{C}^n\atop|a|_2=1}|D^{\nu}(\chi-\widetilde{\chi})(x)(a^{\nu})|_2\nonumber\\
&=:&|\chi-\widetilde{\chi}|_{A_k^i}^{\mu_0}\leq \frac{t\beta}{2}\,,
\end{eqnarray}
where the inequality is a consequence of Cauchy inequality and
$$|\tilde{k}|_2\leq |\tilde{k}|=(\sum_{i=1}^{n}|k_i|+|l|)\,.$$
Also applying Cauchy inequality to (\ref{[chi]}), we have
\begin{equation}
\big|[\widetilde{\chi}]_k\big|_{A_k^i}\leq
|\widetilde{\chi}|_{A_k^i}\leq M_0,
\end{equation}
and
\begin{equation}\label{[chi]_k}
\max_{0\leq\mu\leq\mu_0+1}\big|D^{\mu}[\chi]_k\big|_{A_k^i}\leq
\max_{0\leq\mu\leq\mu_0+1}|D^{\mu}\chi|_{A_k^i}\leq tM_1.
\end{equation}

Now from \eqref{three-ineq}, \eqref{|chi|-1}, \eqref{[chi]_k} and
$$
\min_{x\in
A_k^i}\max_{0\leq\mu\leq\mu_0}\big|D^{\mu}[\chi]_k(x)\big|_2\geq\min_{x\in
\mathcal {K}}\max_{0\leq\mu\leq\mu_0}\big|D^{\mu}[\chi]_k(x)\big|_2\geq
t\beta(\omega,\mathcal {K})\,,
$$
it permits the application of \Cref{Rlem:3}, with
$g=[\chi]_k|_{\widetilde{B}}$, $\tilde{g}=[\widetilde{\chi}]_k$,
$B=\widetilde{B}$. As a result we obtain
\begin{equation}\label{measure}
m\{x\in A_k^i~\big{|}~|[\widetilde{\chi}]_k(x)|\leq \varepsilon
\}\leq M\varepsilon^{\frac{1}{\mu_0}}\,,~
\mbox{for}~0<\varepsilon\leq\frac{t\beta}{2\mu_0+2}\,,
\end{equation}
where
$$M=3(2\pi
e)^{\frac{n}{2}}\frac{(\mu_0+1)^{\mu_0+2}}{(\mu_0+1)!}d^{n-1}(n^{-\frac{1}{2}}+2d+
\theta^{-1}d)(t\beta)^{-1-\frac{1}{\mu_0}}tM_1~.$$

On the other hand,
\begin{eqnarray}\label{|chi|}
\mathcal {K}\backslash\mathcal
{H}(\widetilde{\chi})&=&\bigcup_{k\neq 0}\Big\{x\in \mathcal
{K}~\big{|}~|\langle k,t\widetilde{\omega}(x) \rangle+2\pi
l(t\widetilde{\omega},k)|< \frac{\pi}{2}\,
\frac{t\gamma}{|k|^{\tau}}\Big\}\nonumber\\
&=&\bigcup_{k\neq 0}~\bigcup_{i=-\varphi_k}^{\varphi_k}\Big\{x\in
A_k^i~\big{|}~|\langle k,t\widetilde{\omega}(x) \rangle+2\pi l_i|<
\frac{\pi}{2}\, \frac{t\gamma}{|k|^{\tau}}\Big\}\nonumber\\
&=&\bigcup_{k\neq 0}~\bigcup_{i=-\varphi_k}^{\varphi_k}\Big\{x\in
A_k^i~\big{|}~|[\widetilde{\chi}]_k|< \frac{\pi}{2}\,
\frac{t\gamma}{|k|^{\tau}|\tilde{k}|}\Big\}\nonumber\\
&\subseteq&\bigcup_{k\neq
0}~\bigcup_{i=-\varphi_k}^{\varphi_k}\Big\{x\in
A_k^i~\big{|}~|[\widetilde{\chi}]_k|< \frac{\pi}{2}\,
\frac{t\gamma}{|k|^{\tau+1}}\Big\}\,.\nonumber
\end{eqnarray}
By means of (\ref{measure}) we get
\begin{eqnarray}\label{m(k)}
m(\mathcal {K}\backslash\mathcal
{H}(\widetilde{\chi}))&\leq&\sum_{k\neq
0}\sum_{i=-\varphi_k}^{\varphi_k}M\Big(\frac{\pi}{2}\frac{t\gamma}{|k|^{\tau+1}}\Big)^{\frac{1}{\mu_0}}\nonumber\\
&=&M\sum_{k\neq
0}(2\varphi_k+1)(\frac{\pi}{2}t\gamma)^{\frac{1}{\mu_0}}\Big(\frac{1}{|k|^{\tau+1}}\Big)^{\frac{1}{\mu_0}}\,,\nonumber
\end{eqnarray}
provided that
\begin{equation}\label{cond:1}
\frac{\pi}{2}\frac{t\gamma}{|k|^{\tau+1}}\leq \frac{\pi}{2}t\gamma
\leq \frac{t\beta}{2\mu_0+2}\,,
\end{equation}
for $k\in \mathbb{Z}^n\backslash\{0\}$. Due to (\ref{|l|}), we get
$$m(\mathcal {K}\backslash\mathcal
{H}(\widetilde{\chi}))\leq
M\Big(\frac{\pi}{2}t\gamma\Big)^{\frac{1}{\mu_0}}\Big(\frac{M_0}{\pi}+1\Big)\sum_{k\neq
0}\Big(\frac{1}{|k|^{\tau+1-\mu_0}}\Big)^{\frac{1}{\mu_0}}.$$
Because $\tau>(n+1)\mu_0$, series $\sum\limits_{k\neq
0}|k|^{-\frac{\tau+1-\mu_0}{\mu_0}}$ converges. Consequently,
\begin{eqnarray}
m(\mathcal {H}(\widetilde{\chi}))&\geq&m(\mathcal
{K})-\widetilde{M}\cdot(t\gamma)^{\frac{1}{\mu_0}}\cdot(t\beta)^{-1-\frac{1}{\mu_0}}\cdot
t\nonumber\\
&=&m(\mathcal
{K})-\widetilde{M}\gamma^{\frac{1}{\mu_0}}\beta^{-1-\frac{1}{\mu_0}}\,,\nonumber
\end{eqnarray}
where
\begin{equation}
\widetilde{M}=3(2\pi
e)^{\frac{n}{2}}\frac{(\mu_0+1)^{\mu_0+2}}{(\mu_0+1)!}d^{n-1}(n^{-\frac{1}{2}}+2d+\theta^{-1}d)\,\Big(\frac{\pi}{2}\Big)
^{\frac{1}{\mu_0}}\,\Big(\frac{M_0}{\pi}+1\Big)M_1 \sum\limits_{k\in
\mathbb{Z}^n\backslash \{0\}} |k|^{-\frac{\tau+1-\mu_0}{\mu_0}}~.
\end{equation}
Thus $m(\mathcal {H}(\widetilde{\chi}))\geq 0,$ when
\begin{equation}\label{gamma}
\gamma\leq \Big(\frac{m(\mathcal
{K})\beta^{1+\frac{1}{\mu_0}}}{\widetilde{M}}\Big)^{\mu_0}.
\end{equation}

It remains to show (\ref{cond:1}) through (\ref{gamma}). Using
$m(\mathcal {K})\leq (2d)^n$ and $M_0\geq 2\pi$, we obtain
\begin{equation}
\gamma^{\frac{1}{\mu_0}}<\frac{(2d)^n}{\widetilde{M}}\beta^{1+\frac{1}{\mu_0}}<\frac{(2d)^n
\cdot\beta^{1+\frac{1}{\mu_0}}}{3(2\pi
e)^{\frac{n}{2}}(\mu_0+1)d^{n+1}\cdot 2d\cdot M_1}~.
\end{equation}
By means of $6(2\pi e)^{\frac{n}{2}}>2^n\pi$,
$(\mu_0+1)^{\mu_0}>\mu_0+1$, and
$\beta\leq\max\limits_{0<\mu\leq\mu_0+1}\big|D^{\mu}[\chi]_k\big|_{A_k^i}\leq
M_1$, we have
$$\gamma<\frac{\beta}{\pi (\mu_0+1)}.$$
That completes the proof.

\end{proof}

\end{appendix}

\end{document}